\documentclass[12pt]{amsart}
\usepackage{comment} 
\usepackage{mathrsfs} 
\usepackage[margin=1in]{geometry} 
\usepackage{subcaption} 
\usepackage{hyperref} 
\hypersetup{colorlinks=true, linkcolor=blue, filecolor=blue, urlcolor=blue, citecolor=blue, pdftitle={Splitting of abelian varieties in motivic stable homotopy category}} 

\usepackage{amsmath} 
\usepackage{amsfonts} 
\usepackage{amssymb} 
\usepackage{graphicx} 
\usepackage{tikz-cd} 
\usepackage{amsthm} 
\usepackage[all]{xy} 
\usepackage{adjustbox} 
\usepackage{mathtools} 
\usepackage{hyperref} 

\newtheorem{theorem}{Theorem}[section]
\newtheorem{lemma}[theorem]{Lemma}
\newtheorem{proposition}[theorem]{Proposition}
\newtheorem{corollary}[theorem]{Corollary}
\newtheorem{definition}[theorem]{Definition}

\theoremstyle{definition}
\newtheorem{remark}[theorem]{Remark}
\newtheorem{notation}[theorem]{Notation}
\newtheorem{example}[theorem]{Example}

\numberwithin{equation}{section}

\newcommand{\Z}{\mathbb{Z}}
\newcommand{\R}{\mathbb{R}}
\newcommand{\C}{\mathbb{C}}
\newcommand{\Q}{\mathbb{Q}}

\newcommand{\PP}{\mathfrak{P}}

\newcommand{\F}{\mathbb{F}}
\newcommand{\p}{\mathfrak{p}}
\newcommand{\D}{\mathfrak{d}}

\title{Splitting of abelian varieties in motivic stable homotopy category}
\author{Haoyang Liu}
\address{Department of Mathematics, University of California, Santa Barbara, CA, USA}
\email{haoyangliu@ucsb.edu}
\thanks{}

\begin{document}
\begin{abstract}
    In this paper, we discuss the motivic stable homotopy type of abelian varieties. For an abelian variety over a perfect field $k$, a top-dimensional cell always splits off in the motivic stable homotopy category $\text{SH}(k)$. Let $k=\R$. If $X$ is a real abelian variety of positive dimension $g$, there is, for $\Z[1/(2g)!]\subset\Lambda\subset\Q$, a concrete splitting in $\text{SH}(\R)_\Lambda$ determined by the Deninger--Murre decomposition of the motive of $X$ and the real points $X(\R)$. With rational coefficients, we also discuss its refinement from the viewpoint of K\"unnemann's Lefschetz decomposition.
\end{abstract}
\maketitle
\tableofcontents

\section{Introduction}
In classical homotopy theory, stable splitting is an interesting phenomenon in the stable homotopy category $\text{SH}$ \cite{MR1867354}. Any stable splitting induces direct-sum decompositions in every representable generalized (co)homology theory. By naturality, stable cohomology operations, including Steenrod operations, commute with the resulting projectors; see \cite[Chapter III]{Adams1974}. In the motivic setting, interesting examples have been found, such as the following theorem of R\"ondigs \cite{10.1093/qmath/hap005}:

\begin{theorem}\label{theta}
    Let $k$ be a field and $X$ be a smooth projective curve over $k$ with a rational point
    $x_0:S^{0,0}=\text{Spec}(k)_+\to X_+$. There is a splitting
    \begin{equation*}
        X_+\simeq S^{0,0}\vee \mathbb{J}(X)\vee S^{2,1}
    \end{equation*}
    in the motivic stable homotopy category $\text{SH}(k)$ if and only if $X$ admits a theta characteristic.
\end{theorem}

Recall that a \emph{theta characteristic} of a smooth curve $X$ is a line bundle $L\to X$ such that $L\otimes L\cong T_X$, where $T_X$ denotes the tangent bundle of $X$. The splitting in Theorem \ref{theta} lifts the corresponding decomposition of the motive of $X$, and $\mathbb{J}(X)$ maps to the Jacobian part. Furthermore, a theta characteristic supplies the orientation needed in R\"ondigs' construction.

The construction of the splitting in Theorem \ref{theta} is based on Spanier--Whitehead duality in $\text{SH}(k)$ as well as a connectivity theorem of Voevodsky (cf.\ \cite{10.1093/qmath/hap005}). Consequently, the values of any (co)homology theory representable in $\text{SH}(k)$ decompose accordingly.
The same type of splitting can also be achieved for smooth projective varieties over $k$ with trivial tangent bundle, such as abelian varieties. However, just as in the case of a curve, this argument only splits off the top cell. For an abelian variety with dimension greater than 1,
this splitting is not so satisfactory compared to the well-known decomposition of motives of abelian varieties \cite{Deninger1991}. The goal of this paper is therefore to improve this splitting after passing to suitable localizations of $\text{SH}(k)$.

Related work of Liu \cite{Liu24} studies the integral problem of splitting off the top cell of a smooth projective surface by means of motivic duality and characteristic-class obstructions. By contrast, the present paper treats abelian varieties in arbitrary dimension and obtains finer decompositions after localization by combining the plus part, controlled by Chow motives, with the minus part, controlled over $\R$ by real realization.

\subsection{Main theorem}

After inverting $2$, Morel's idempotents give the plus--minus decomposition of $\text{SH}(k)$ \cite[Section 6]{Mor10}; every object therefore has a plus part and a minus part. One may regard the plus part as the ``oriented'' part and the minus part as the ``unoriented'' part. The plus part is related to motives by the motivization functor $M$ \cite{Bachmann_20182,bondarko2019infinite}, and rationally Cisinski and D\'eglise identify $\text{SH}(k)^+_{\Q}$ with $\text{DM}(k)_{\Q}$ \cite[Section 16.2]{Cisinski_2019}. Over localizations of $\Z$, conservativity alone does not lift projectors. We instead verify that the Deninger--Murre projectors lie in the image of motivization and that the kernel on endomorphisms is nilpotent; the associated direct-sum decomposition therefore lifts as well. For the minus part over $\R$, Bachmann's real-\'etale comparison identifies the relevant localization with the classical stable homotopy category \cite[Theorem 35]{Bachmann_20181}.

The decomposition of the rational Chow motive of an abelian variety is due to Deninger and Murre \cite{Deninger1991}. The topology of a real abelian variety and the number of its connected components are described by Gross--Harris and Huisman \cite{ASENS,Huisman1994}. Combining these results gives the following splitting. Here $\text{SH}(\R)_\Lambda$ denotes the $\Lambda$-linear motivic stable homotopy category, $S^{p,q}_\Lambda$ denotes the bigraded motivic sphere, and $\simeq$ denotes an isomorphism in the indicated stable homotopy category. We write
\[
 S^{p,q,\pm}_\Lambda:=S^{p,q}_\Lambda\wedge\mathbb I^\pm_\R
\]
for its plus and minus summands.
\begin{theorem}\label{abeliansplitting}
    Let $X$ be a real abelian variety of dimension $g\geq1$, and let $\Z[1/(2g)!]\subset\Lambda\subset\Q$. There are objects $\mathbb H^r(X)\in\text{SH}(\R)^+_\Lambda$, for $0\leq r\leq2g$, such that $M(\mathbb H^r(X))\simeq M^r_{gm}(X)$ and
    \begin{equation*}
        X_+\simeq
        \left(\bigvee_{r=0}^{2g}\mathbb H^r(X)\right)
        \vee
        \left(\bigvee^{n(X)}\bigvee_{i=0}^{g}\bigvee^{\binom gi}S^{i,0,-}_\Lambda\right).
    \end{equation*}
    Here $n(X)$ is the number of connected components of $X(\R)$; under the additional hypotheses of Theorem \ref{components1}, it is given by Huisman's formula there. With rational coefficients and a choice of polarization, the first parenthesis further refines to K\"unnemann's Lefschetz--primitive splitting:
    \begin{equation*}
        \bigvee_{i=0}^{g-1}\left(\bigvee_{k=0}^{\lfloor i/2\rfloor}
        (S^{2k,k}_{\Q}\vee S^{2(k+g-i),k+g-i}_{\Q})\wedge\mathbb J_{i-2k}(X)\right)
        \vee\bigvee_{k=0}^{\lfloor g/2\rfloor}S^{2k,k}_{\Q}\wedge\mathbb J_{g-2k}(X),
    \end{equation*}
    where $\mathbb J_i(X)\in\text{SH}(\R)^+_{\Q}$ and $M(\mathbb J_i(X))\simeq P^i(X)$.
\end{theorem}
The formula for $n(X)$ used here is due to Huisman and is expressed in terms of the endomorphism rings of $X$ and its complexification, together with the ramification invariants of the semilinear lattice defining the real form; its proof reduces to a computation in $\Z/2$-group cohomology.

\subsection{Outline of the paper}

In Section \ref{prem}, we will revisit the construction of $\text{SH}(k)$ and $\text{DM}(k)$, and also some basic properties of these two categories and their localizations. 
In Section \ref{section3}, we will revisit the splitting of abelian varieties in $\text{SH}(k)$ using the same idea as Theorem \ref{theta}.
In Section \ref{coe}, we will introduce the construction of $\text{SH}(k)^+$ and $\text{SH}(k)^-$ after inverting 2 and some useful and important results to understand them. 
In Section \ref{motive}, we will revisit the Chow motives and results of decomposition of the Chow motive of an abelian variety, which are related to the splitting in the plus part.
In Section \ref{final}, we will present the main result and some examples.
In Section \ref{appendix}, which is organized as an appendix, we revisit the Fourier transform on abelian varieties and the topology of the real points of a real abelian variety. These results lead to the classical stable splitting used to describe the minus part.

\subsection{Acknowledgements}
The author first wants to thank Aravind Asok for introducing him to this topic and this project and also for his helpful advice. He is also grateful to Eric Friedlander, Tianle Liu and Masoud Zargar for helpful discussions. Special thanks go to Yeqin Liu and Keyao Peng for their helpful suggestions for a draft of this paper.

\section{Preliminaries}\label{prem}
In this section we shall work over a fixed perfect base field $k$. We will follow \cite{Bachmann_20182} and \cite{bondarko2019infinite}.

\begin{notation}

For categories $\mathcal{C}, \mathcal{D}$ we write $\mathcal{C}\subset \mathcal{D}$ if $\mathcal{C}$ is a full subcategory of $\mathcal{D}$.

For a category $\mathcal{C}$ and $X, Y\in\text{Obj}(\mathcal{C})$, the set of $\mathcal{C}$-morphisms from $X$ to $Y$ will be denoted by $\mathcal{C}(X,Y)$.

In the following $\mathcal{C}$ is a triangulated category closed with respect to all small coproducts.

\end{notation}

\begin{definition}\cite{bondarko2019infinite}
    We will say that an object $M$ of $\mathcal{C}$ is \emph{compact} whenever the functor $\mathcal{C}(M,-)$ respects coproducts. A class $\{C_i\}\subset\text{Obj}(\mathcal{C})$ generates a subcategory $\mathcal{D}\subset \mathcal{C}$ as a \emph{localizing}
    subcategory if $\mathcal{D}$ equals the smallest full strict triangulated subcategory of $\mathcal{C}$ that is closed with respect to small coproducts and contains $\{C_i\}$. Moreover, $\{C_i\}$ \emph{compactly generates} $\mathcal{C}$ 
    if $\{C_i\}$ is a set, all $C_i$ are compact in $\mathcal{C}$ and $\{C_i\}$ generates $\mathcal{C}$ as its own localizing subcategory. We will say that $\mathcal{C}$ is compactly generated whenever there exists some set of compact generators like above. 
    We will denote the full subcategory of compact objects of $\mathcal{C}$ by $\mathcal{C}^c$.
\end{definition}

We also recall some basics on ``localizing coefficients'' in a triangulated category. Below $S\subset\Z$ will always be a set of prime numbers and the ring 
$\Z[S^{-1}]$ will be denoted by $\Lambda$.

\begin{proposition}\cite{bondarko2019infinite}
    Assume that $\mathcal{C}$ is compactly generated by a small subcategory $\mathcal{C}'$. Denote the localizing subcategory of $\mathcal{C}$ by $\mathcal{C}_{S-tors}$ and it is (compactly) generated by $\text{Cone}(c'\stackrel{\times s}{\to}c')$ for $c'\in\text{Obj}(\mathcal{C}')$, $s\in S$.
    
    Then the following statements are valid.

    1.\:$\mathcal{C}_{S-tors}$ also contains all cones of $c\stackrel{\times s}{\to}c$ for $c\in\text{Obj}(\mathcal{C})$ and $s\in S$.

    2.\:The Verdier quotient category $\mathcal{C}_{\Lambda}=\mathcal{C}/\mathcal{C}_{S-tors}$ exists. The localization functor $l:\mathcal{C}\to \mathcal{C}_{\Lambda}$ respects all coproducts and converts compact objects into compact ones.
    Moreover, $\mathcal{C}_{\Lambda}$ is generated by $l(Obj(\mathcal{C}'))$ as a localizing subcategory.

    3.\:For any $c\in Obj(\mathcal{C}), c'\in Obj(\mathcal{C}')$, we have $\mathcal{C}_{\Lambda}(l(c'),l(c))\cong \mathcal{C}(c',c)\otimes_{\Z}\Lambda$.

    4.\:$l$ possesses a right adjoint $G$ which is a full embedding functor. The essential image of $G$ consists of those $M\in Obj(\mathcal{C})$ such that $s\cdot id_M$ is an automorphism for any $s\in S$, i.e. $G(\mathcal{C}_{\Lambda})$ is the maximal full $\Lambda$-linear subcategory
    of $\mathcal{C}$.
\end{proposition}
\begin{proof}
    See Proposition 5.6.2(I) of \cite{bondarko2019torsion}.
\end{proof}

\begin{remark}
    For a triangulated category $\mathcal{C}$, its $\Lambda$-linear version will be denoted by $\mathcal{C}_{\Lambda}$.
\end{remark}

More generally, if $u:\Sigma^{a,b}\mathbb I\to\mathbb I$ is a homogeneous endomorphism of the tensor unit, we write $\mathcal C[u^{-1}]$ for the tensor Bousfield localization in which $u_E:\Sigma^{a,b}E\to E$ becomes invertible for every object $E$. Equivalently, its kernel is the localizing tensor ideal generated by the cofiber of $u$, and its local objects are precisely the $u$-periodic objects. We shall use this notation for $u=\eta$ and $u=\rho$ in Section \ref{coe}.

Now we recall the construction of $\text{SH}(k)$ and $\text{DM}(k)$ \cite{RONDIGS2008689}. Let $Sm_k$ be the category of smooth schemes over the perfect field $k$,
and $Cor_k$ be the category whose objects are the smooth schemes and whose morphisms are the finite correspondences. We write $Shv_k$ (respectively $Shv^{tr}_k$) for the categories
of Nisnevich sheaves. Write $R:Sm_k\to Shv_k$ and $R_{tr}:Cor_k\to Shv^{tr}_k$ for the functors sending an object to the sheaf it represents.

There is a natural functor $M:Sm_k\to Cor_k$ defined by $M(X)=X$ and $M(f)=\Gamma_f$, the graph of $f$. This induces a functor $U:Shv^{tr}_k\to Shv_k$ via $(UF)(X)=F(M(X))$. The functor $U$ admits a left adjoint
$M:Shv_k\to Shv^{tr}_k$, which is the unique colimit-preserving functor satisfying $M(R(X))=R_{tr}(X)$. Let $Shv_{k,*}$ denote the category of pointed sheaves. There is a functor $R_+:Sm_k\to Shv_{k,*}$
defined by adjoining a disjoint base point. The objects in the essential image $U(Shv^{tr}_k)$ are canonically pointed by zero and this gives rise to a new adjunction: $\tilde{M}:Shv_{k,*}\leftrightarrows Shv^{tr}_k:U$.

Passing to simplicial objects and extending $M$ and $U$ levelwise, we obtain an adjunction $\tilde{M}:\Delta^{op}Shv_{k,*}\leftrightarrows \Delta^{op}Shv^{tr}_k:U$. We denote by $Spt_k$ the category of 
$S^{2,1}:=S^1\wedge\mathbb{G}_m$-spectra in $\Delta^{op}Shv_{k,*}$ and by $Spt^{tr}_k$ the category of $M(S^{2,1})$-spectra in $\Delta^{op}Shv^{tr}_{k}$. The adjunction extends, and we have the following commutative diagram.
\begin{equation*}
    \begin{tikzcd}
        Sm_k \arrow[r,"M"] \arrow[d,"R_+"] & Cor_k \arrow[d,"R_{tr}"]\\
        Shv_{k,*} \arrow[r, bend left=10, "\tilde{M}"] \arrow[d] & Shv^{tr}_k \arrow[l, bend left=10, "U"] \arrow[d]\\
        \Delta^{op}Shv_{k,*} \arrow[r, bend left=10, "\tilde{M}"] \arrow[d,"\Sigma^{\infty}"] & \Delta^{op}Shv^{tr}_k \arrow[l, bend left=10, "U"] \arrow[d,"\Sigma^{\infty}"]\\
        Spt_k \arrow[r, bend left=10, "M"]  & Spt^{tr}_k \arrow[l, bend left=10, "U"] 
    \end{tikzcd}
\end{equation*} 

\begin{definition}
    One may put the projective local model structures on the lower square and then the adjunctions become Quillen adjunctions, so they pass through localization. Contracting the affine line yields the $\mathbb{A}^1$-local model structures. The homotopy 
category of $Spt_k$ (in this model structure) is denoted by $\text{SH}(k)$ and is called the motivic stable homotopy category. Similarly the homotopy category of $Spt^{tr}_k$ is denoted by $\text{DM}(k)$.
\end{definition}
 
$\text{DM}(k)$ is a larger version of Voevodsky's category and is often called ``the big derived category of motives'' \cite{RONDIGS2008689}. We have the following commutative diagram:
\begin{equation*}
    \begin{tikzcd}
        Sm_k \arrow[r,"M"] \arrow[d,"\Sigma^{\infty}(\bullet_+)"] & Cor_k \arrow[d]\\
        \text{SH}(k) \arrow[r, bend left=10, "M"]& \text{DM}(k) \arrow[l, bend left=10, "U"]
    \end{tikzcd}
\end{equation*} 
In all its incarnations, $M$ is a symmetric monoidal functor, and one may regard $M$ as a motivic analogue of ``the Hurewicz homomorphism''.

The categories $\text{SH}(k)$ and $\text{DM}(k)$ are tensor triangulated categories. In the following proposition we recall some of their basic properties.
\begin{proposition}\cite{bondarko2019infinite}\label{tensor}
If we work on the smooth $k$-varieties $\text{SmVar}_k$, there exist functors $\text{SmVar}\to\text{SH}(k): X\mapsto \Sigma^{\infty,\infty} X_+$ and $M_{gm}:\text{SmVar}\to\text{DM}(k)$ (and $M_{gm}(X)$ actually falls in Voevodsky's triangulated category of motives $\text{DM}_{gm}(k)$ \cite{59}). 
By the construction above, we have $M:\text{SH}(k)\to\text{DM}(k)$. This functor is exact, respects coproducts, and preserves compact objects. We have $M(\Sigma^{\infty,\infty} X_+)\cong M_{gm}(X)$ for any $X\in\text{SmVar}$.
And also as mentioned in the construction above, $M$ has a right adjoint $U$ that respects coproducts. The categories $\text{SH}(k)$ and $\text{DM}(k)$ are both compactly generated.
\end{proposition}

Note that Proposition \ref{tensor} is also true for the case with coefficients $\Lambda$.

\section{Splitting in the motivic stable homotopy category}\label{section3}
In this section we keep the base $S=\operatorname{Spec}(k)$, with $k$ perfect, and recall the splitting of abelian varieties in $\text{SH}(k)$ using the techniques of \cite{10.1093/qmath/hap005}. The duality theorem recalled below is valid over more general Noetherian finite-dimensional bases, but that generality is not needed in this paper. Denote by $\text{SH}(S)$ the motivic stable homotopy category over $S$.
A \emph{motivic space} $A$ over $S$ is a presheaf on the site $\text{Sm}_S$ (the category of smooth separated $S$-schemes) with values in the category 
of simplicial sets. A \emph{motivic spectrum} $E$ over $S$ consists of a sequence $(E_0,E_1,...)$ of pointed motivic spaces over $S$, together with 
a sequence of structure maps $\sigma^{E}_{n}:E_n\wedge S^{2,1}\to E_{n+1}$, where the smash product of pointed simplicial presheaves is defined 
sectionwise. Here $S^{2,1}=\mathbb{A}^1_S/(\mathbb{A}^1_S\setminus\{ 0\})$ denotes the Thom space of the trivial line bundle over $S$. A smooth $S$-scheme $x:X\to S$ determines a representable pointed simplicial presheaf by adjoining a disjoint base point. 
The associated $S^{2,1}$-suspension spectrum is denoted as $\Sigma^{\infty, \infty} X_+$ (or simply $X_+$). Its $n$-th structure map is the identity on $X_+\wedge S^{2n,n}$, where $S^{2n,n}=S^{2(n-1),n-1}\wedge S^{2,1}$.

The category $\text{SH}(S)$ is closed symmetric monoidal under the smash product $E\wedge F$, with unit object $\mathbb{I}_S:=S_+=S^{0,0}$.

If $f:S\to S'$ is a morphism of base schemes, there is an adjoint pair of functors:
\begin{equation*}
    f_*:\text{SH}(S)\rightleftharpoons\text{SH}(S'):f^*.
\end{equation*}
If the morphism $f$ is smooth, then $f^*$ admits a left adjoint $f_\sharp:\text{SH}(S)\to\text{SH}(S')$, and the projection formula holds.

Let $p:V\to S$ be a vector bundle over $S$, with zero section $z:S\to V$. Let $\text{Th}(V)$ denote the $S^{2,1}$-suspension spectrum of the pointed (simplicial)
presheaf that assigns to each $U\to S\in\text{Sm}_S$ the quotient set $\text{Hom}_{\text{Sm}_S}(U,V)/Hom_{\text{Sm}_S}(U,V\setminus z(S))$. This is called the \emph{Thom spectrum}
of the bundle $p:V\to S$. In the special case where $V\cong\mathbb{A}^n\to S$ is a trivial vector bundle, the Thom spectrum $\text{Th}(\mathbb{A}^n)$ is simply the $n$-fold smash product of the 
$S^{2,1}$-suspension spectrum of $S^{2,1}$ itself. Let $S^{1,0}$ be the $S^{2,1}$-suspension spectrum of the constant pointed simplicial presheaf sending every $U\to S\in\text{Sm}_S$
to $\Delta^1/\partial\Delta^1$. The relation $S^{2,1}\simeq S^{1,0}\wedge (\mathbb{A}^1_S\setminus \{ 0\}, 1)$ shows that $S^{1,0}$ is invertible under the smash product as well.
Denote $S^{1,1}=(\mathbb{A}^1\setminus \{0\}, 1)$. Then for every pair $(p,q)$ of integers there is a bigraded motivic sphere
\begin{equation*}
    S^{p,q}:= S^{p-2q,0}\wedge S^{2q,q}\in\text{SH}(S),
\end{equation*}
which is invertible with respect to the smash product.

The construction of the splitting relies on a special case of a connectivity theorem established by Voevodsky \cite{10.1093/qmath/hap005}.

\begin{theorem}\label{Conn}
    Let $S=\text{Spec}(k)$ be the spectrum of a field and $p,q\in\mathbb{Z}$. Then
    \begin{equation*}
        \text{Hom}_{\text{SH}(S)}(S^{0,0}, S^{p,q})=0
    \end{equation*}
    whenever $p>q$.
\end{theorem}

Next, we describe the Spanier-Whitehead dual of a smooth projective scheme $x:X\to S$ with tangent bundle 
$\mathcal{T}(x)\to X$. Let $\mathbb{I}_S\in\text{SH}(S)$ denote the unit for the smash product in the motivic stable
homotopy category, given by the $S^{2,1}$-suspension spectrum of the zero sphere $S_+$. The \emph{Spanier-Whitehead dual}
of an object $E\in\text{SH}(S)$ is defined as the internal Hom $\mathcal{D}(E):=\text{SH}(S)(E,\mathbb{I}_S)$. For example, the Spanier-Whitehead dual of
$S^{p,q}$ is $S^{-p,-q}$. The following result concerning Spanier-Whitehead duality is proven in \cite{HU2005609} in the case where $S=\text{Spec}(k)$. The general 
case is addressed in \cite{AST_2007__314__R1_0}.

\begin{theorem}\label{Dual}
    Let $S$ be a base scheme, and let $x:X\to S$ be a smooth projective morphism. There is an isomorphism
    \begin{equation*}
        \mathcal{D}(X_+)\sim x_{\sharp}(\text{Th}(-[\mathcal{T}(x)]))
    \end{equation*}
    in $\text{SH}(S)$, where $[\mathcal{T}(x)]\in K^0(X)$ is the class of the tangent bundle of $x:X\to S$.
\end{theorem}

If $A$ is a retract of $B$ in a stable homotopy category such as $\text{SH}(S)$, then $A$ is, in fact, a summand of $B$. Given $X\in\text{Sm}_S$
with the structural map $x:X\to S$, observe that $S^{0,0}=S_+$ is a retract of $X_+$ in $\text{SH}(S)$ if there exists a morphism $x_0:S^{0,0}=S_+\to X_+$
in $\text{SH}(S)$ such that $x_+\circ x_0$ is the identity element in $\pi_{0,0}(S^{0,0})=\text{Hom}_{\text{SH}(S)}(S^{0,0}, S^{0,0})$. A morphism
$x_0$ satisfying this condition is called a \emph{rational point up to stable homotopy}. Every rational point is also a rational point up to stable homotopy.
The converse need not hold. Colliot-Th\'el\`ene--Coray construct a smooth projective geometrically integral conic-bundle surface $Y$ over a $p$-adic field $k$ with closed points $P_2$ and $P_3$ of residue degrees $2$ and $3$, respectively, but with $Y(k)=\varnothing$ \cite[p.~303 and \S~5]{ColliotTheleneCoray1979}. Thus $[P_3]-[P_2]$ is a zero-cycle of degree $1$. Since $k$ is infinite and perfect and $\operatorname{char}(k)\neq2$, the criterion of Asok--Haesemeyer \cite[Theorem~4.4.5]{asok20110thstablea1homotopysheaf} shows that $Y$ has a rational point up to stable $\mathbb A^1$-homotopy in the above sense. Because $Y(k)=\varnothing$, this stable point is not induced by an ordinary $k$-rational point. All applications below nevertheless use an actual rational point, namely the identity section of an abelian variety.

The following is the standard stable retract--duality argument; compare \cite[Corollary 5.2]{10.1093/qmath/hap005}.
\begin{theorem}\label{Shift1}
    Let $x:X\to S$ be a smooth projective connected scheme over $S=\text{Spec}(k)$ of dimension $d>0$ with a rational point up to stable
    homotopy $x_0:S^{0,0}\to X_+$. Suppose that $x_{\sharp}(\text{Th}(-[\mathcal{T}(x)]))$ is isomorphic to $S^{-2d,-d}\wedge X_+\in\text{SH}(S)$.
    Then $X_+$ splits as
    \begin{equation*}
        X_+\sim S^{0,0}\vee F\vee S^{2d,d}
    \end{equation*}
    for some $F\in\text{SH}(S)$.
\end{theorem}
\begin{proof}
    Write $(X,x_0)$ for the cofiber of $x_0:S^{0,0}\to X_+$ in $\text{SH}(S)$.
    The rational point up to stable homotopy $x_0$ and the structure map $x$ imply that $S^{0,0}$ is a retract of $X_+$. As mentioned above, there is a splitting 
    $X_+\stackrel{\sim}{\longrightarrow}(X,x_0)\vee S^{0,0}$ in $\text{SH}(S)$, given by a morphism $c:X_+\to (X,x_0)$ and the structural map $x_+$. Let $d:(X,x_0)\to(X,x_0)\vee S_+\stackrel{\sim}{\longrightarrow}X_+$
    denote the canonical map. To produce the splitting, it suffices to show that $S^{2d,d}$ is a retract of $(X,x_0)$.\\
    Applying the Spanier-Whitehead duality functor to the morphisms $S^{0,0}\stackrel{x_0}{\longrightarrow}X_+\stackrel{x_+}{\longrightarrow}S^{0,0}$ produces morphisms
    \begin{equation*}
        S^{0,0}\sim\mathcal{D}(S^{0,0})\stackrel{\mathcal{D}(x_0)}{\gets}\mathcal{D}(X_+)\stackrel{\mathcal{D}(x_+)}{\gets}\mathcal{D}(S^{0,0})\sim S^{0,0}
    \end{equation*}
    Then by Theorem \ref{Dual} and our hypothesis, we have the isomorphism $\mathcal{D}(X_+)\sim S^{-2d,-d}\wedge X_+$, and thus, after tensoring with $S^{2d,d}$ there is a diagram
    \begin{equation*}
        S^{2d,d}\stackrel{\varphi}{\longleftarrow}X_+\stackrel{\psi}{\longleftarrow}S^{2d,d}
    \end{equation*}
    which shows that $S^{2d,d}$ is a retract of $X_+$. To obtain the desired result, we need to conclude that the composition
    \begin{equation*}
        S^{2d,d}\stackrel{\varphi}{\longleftarrow}X_+\stackrel{d}{\longleftarrow}(X,x_0)\stackrel{c}{\longleftarrow}X_+\stackrel{\psi}{\longleftarrow}S^{2d,d}
    \end{equation*}
    is the identity. This composition is the image of $\text{id}_{S^{2d,d}}$ under the sequence of maps
    \begin{equation*}
        [S^{2d,d}, S^{2d,d}]\stackrel{\varphi^*}{\longrightarrow}[X_+,S^{2d,d}]\stackrel{d^*}{\longrightarrow}[(X,x_0), S^{2d,d}]\stackrel{c^*}{\longrightarrow}[X_+,S^{2d,d}]\stackrel{\psi^*}{\longrightarrow}[S^{2d,d}, S^{2d,d}],
    \end{equation*}
    where $\text{Hom}_{\text{SH}(S)}(-,-)=[-,-]$. The splitting $X_+\sim S^{0,0}\vee (X,x_0)$ implies that there is a commutative diagram\\
    \begin{equation*}
     \begin{tikzcd}
        \lbrack X_+,S^{2d,d} \rbrack\arrow[r,"\cong"] \arrow[d,"d^*"] & \lbrack S^{0,0}\vee (X,x_0), S^{2d,d}\rbrack \arrow[d,"\cong"]\\
        \lbrack (X,x_0), S^{2d,d} \rbrack & \lbrack S^{0,0},S^{2d,d} \rbrack \oplus \lbrack (X,x_0), S^{2d,d}\rbrack \arrow[l,"\text{pr}"]
     \end{tikzcd}
    \end{equation*}
    The group $[S^{0,0}, S^{2d,d}]$ is trivial by Theorem \ref{Conn}, which implies that $d^*$ is an isomorphism with inverse $c^*$.
\end{proof}

The argument applies to abelian varieties because the tangent bundle of an abelian variety is trivial, hence so is the associated Thom spectrum. So this argument 
in general is true for smooth projective connected schemes with a rational point up to stable homotopy whose Spanier-Whitehead dual is given simply by a shift, and in particular with trivial tangent bundle and with tangent bundle 
whose associated Thom spectrum is trivial.

\begin{corollary}\label{Shift2}
    Let $x:X\to S$ be a smooth projective connected scheme over $S=\text{Spec}(k)$ of dimension $d>0$ with a rational point up to stable
    homotopy $x_0:S^{0,0}\to X_+$. If it also has the property that the Thom spectrum of its tangent bundle is trivial, then there is an isomorphism
    \begin{equation*}
        \mathcal{D}(X_+)\sim S^{-2d,-d}\wedge X_+
    \end{equation*}
    in $\text{SH}(S)$. And then $X_+$ has the corresponding splitting as in Theorem \ref{Shift1}.
\end{corollary}
\begin{proof}
    Theorem \ref{Dual} provides an isomorphism $\mathcal{D}(X_+)\sim x_{\sharp}(\text{Th}(-[\mathcal{T}(x)]))$ in $\text{SH}(S)$. By our assumption, $\text{Th}([\mathcal{T}(x)])\sim X^{2d,d}\in\text{SH}(X)$. Therefore\\
    \begin{equation*}
        x_{\sharp}(\text{Th}(-\mathcal{T}(x)))\sim x_{\sharp}(X^{-2d,-d})\sim x_{\sharp}x^*(S^{-2d,-d})\sim S^{-2d,-d}\wedge X_+
    \end{equation*}
    by the projection formula. Moreover, if the tangent bundle is trivial, then we have $x_{\sharp}(\text{Th}(-\mathcal{T}(x))) = x_{\sharp}(X^{-2d,-d})$
    and the same conclusion follows.
\end{proof}

\begin{remark}\label{projective}
    It is not always true that the tangent bundle of a smooth projective scheme, or its associated Thom spectrum, is trivial. For example, the tangent bundle of a smooth projective curve is non-trivial precisely if the genus of the
    curve is different from one. Nevertheless, a curve with a theta characteristic and a rational point up to stable homotopy has the required splitting by R\"ondigs' theorem. The other simple example is the projective plane $\mathbb{P}^2$,
    a smooth projective scheme whose tangent bundle has a non-trivial Thom spectrum. In fact, if the Thom spectrum of $\mathcal{T}(\mathbb{P}^2)$ was trivial, Corollary \ref{Shift2} would imply that $\mathbb{P}^2_+\sim S^{0,0}\vee S^{2,1}\vee S^{4,2}$
    in the motivic stable homotopy category. However, the algebraic Hopf map $S^{3,2}\sim \mathbb{A}^2\setminus\{ 0\}\to\mathbb{P}^1\sim S^{2,1}$ is stably nontrivial.
\end{remark}

If one is only interested in the $E$-(co)homology of a smooth projective $S$-scheme $X$ whose tangent bundle admits an $E$-orientation, then $E\wedge \mathcal{D}(X_+)\sim S^{-2d,-d}\wedge E\wedge X_+$ by the Thom isomorphism.
In particular, there is no restriction on $X$ if $E$ is an orientable motivic spectrum. Thus the $E$-(co)homology of a smooth projective scheme with a rational point up to stable homotopy always splits off a top-dimensional cell, which is a special case of Poincar\'e duality.

So we can conclude this section by the following theorem:

\begin{theorem}\label{integral}
    Let $X$ be an abelian variety of dimension $g\geq1$ over $k$. Then $X_+$ splits as
    \begin{equation*}
        X_+\sim S^{0,0}\vee F\vee S^{2g,g}
    \end{equation*}
    for some $F\in\text{SH}(k)$.
\end{theorem}

We will see later that passing to some localizing coefficients will make the splitting more complicated and we will get more components in the splitting.

\section{Motivic stable homotopy category with coefficients}\label{coe}
Section \ref{section3} gives an integral splitting of the bottom and top cells of an abelian variety. To obtain a finer decomposition, we now pass to the localized category $\text{SH}(k)_{\Lambda}$ introduced in Section \ref{prem}.

To begin with, we work over a perfect base field $k$ with $\text{char}(k)\neq 2$ and we fix the base $S=\text{Spec}(k)$. Recall $\Lambda$ ($\Z\subset \Lambda\subset\Q$) denotes the chosen localizing
coefficients.

Hopkins-Morel \cite{Mor10} have defined the Milnor-Witt sheaf $K_*^{MW}$. For a field $k$, $K_*^{MW}(k)$ is the graded ring defined by generators and relations:\\
$\bullet$ Generators: for $u\in k^{\times}$ one has the generators $[u]$ in degree $+1$. There is an additional generator $\eta$ in degree $-1$.\\
$\bullet$ Relations:\\
(1)$\eta[u]=[u]\eta$ for all $u\in k^{\times}$;\\
(2)(Twisted additivity) $[uv]=[u]+[v]+\eta[u][v]$ for $u,v\in k^{\times}$;\\
(3)(Steinberg relation) $[u][1-u]=0$ for $u\in k\setminus\{0,1\}$;\\
(4)(Hyperbolic relation) Let $h=2+\eta[-1]$. Then $\eta\cdot h=0$.\\

Let $\text{GW}(k)$ denote the Grothendieck-Witt ring of non-degenerate quadratic forms over $k$; this is the group completion of the monoid (under orthogonal direct sum) of 
non-degenerate quadratic forms over $k$. The hyperbolic form is the rank 2 form $H(x,y)=x^2-y^2$. For $u\in k^{\times}$, we let $q_u$ be the rank one form given by $q_u(x)=ux^2$. We also 
denote $\langle u\rangle=1+\eta[u]\in K^{MW}_0(k)$. In fact this extends to an isomorphism $\text{GW}(k)\cong K^{MW}_0(k)$.

By a theorem of Morel \cite{Mor10}, there is a canonical identification
\begin{theorem}
    $\text{Hom}_{\text{SH}(k)}(\mathbb{I}_k,\mathbb{G}_m^{\wedge *})\cong K^{MW}_*(k)$.
\end{theorem}
In particular, we have $\text{End}_{\text{SH}(k)}(\mathbb{I}_k)\cong K^{MW}_0(k)\cong\text{GW}(k)$.

For $a\in k^{\times}$, there is a corresponding map $[a]:*\to\mathbb{G}_m$ inducing a stable map $[a]:\mathbb{I}_k\to\mathbb{G}_m\in\text{SH}(k)$.

The generator $\eta$ has a simple geometric description: it is the stabilization of $\pi:\mathbb{A}^2\setminus \{ 0\} \to\mathbb{P}^1$ in the motivic stable homotopy category, which sends a point $(x,y)$ to a line $[x:y]$ through it.
This map $\pi$ is called the algebraic Hopf map, as the map on $\C$-points is homotopy equivalent to the classical Hopf map $\alpha:S^3\to S^2$. On the $\R$ points, $({\mathbb{A}}^2 \setminus \{ 0 \})(\R)=\R^2 \setminus\{ 0\}\sim S^1$, $\mathbb{P}^1(\R)=\mathbb{RP}^1\sim S^1$ and $({\mathbb{A}}^2 \setminus \{ 0 \})(\R)\to\mathbb{P}^1(\R)$ is homotopy equivalent to the map $\times 2: S^1\to S^1$.

Now fix the base field to be $k=\R$. Sending $X\in\text{Sm}_k$ to the space of $\R$-points $X(\R)$ extends to a real realization functor
\begin{equation*}
    R_{\R}:\text{SH}(\R)\to\text{SH},
\end{equation*}
where $\text{SH}$ is the classical stable homotopy category and $R_{\R}(\Sigma^{\infty,\infty} X_+)=\Sigma^{\infty}X(\R)_+$. Here, $X\to X(\R)$ assigns a smooth scheme over $\R$ its set of real points 
with the classical (Euclidean) topology.

Recall from Section \ref{section3} that the unit in $\text{SH}(k)$ is the motivic sphere spectrum 
\begin{equation*}
    \mathbb{I}_k:=\Sigma^{\infty,\infty}\text{Spec}(k)_+=(\text{Spec}(k)_+,T,T^{\wedge2},...),
\end{equation*}
where we denote $S^{2,1}$ by $T$.

Now assume $2\in\Lambda^\times$. The involution on $T\wedge T$ exchanging the two factors defines an involution $\tau:\mathbb{I}_k\to\mathbb{I}_k$. Morel's plus--minus decomposition \cite[Section 6]{Mor10} is obtained from the two idempotents $(1+\tau)/2$ and $(1-\tau)/2$. Since $\text{SH}(k)_\Lambda$ is idempotent complete, they give a decomposition
\begin{equation*}
    \mathbb{I}_k\simeq\mathbb{I}^+_k\oplus\mathbb{I}^-_k.
\end{equation*}
We use the convention that $\tau$ acts as the identity on the plus summand and as minus the identity on the minus summand. Since $\mathbb I_k$ is the tensor unit, this induces a categorical equivalence
\begin{equation*}
    \text{SH}(k)_{\Lambda}\simeq\text{SH}(k)^{+}_{\Lambda}\times\text{SH}(k)^{-}_{\Lambda}.
\end{equation*}
Every object $E\in\text{SH}(k)_{\Lambda}$ therefore decomposes as $E\simeq E^+\oplus E^-$.

Morel's computation of the endomorphisms of the motivic sphere identifies $\tau$ with $\langle-1\rangle$ \cite[Section 6]{Mor10}. Thus
\begin{equation*}
    \tau=1+\eta[-1],\qquad \frac{1+\tau}{2}=\frac h2,\qquad
    \frac{1-\tau}{2}=-\frac{[-1]\eta}{2}.
\end{equation*}

Let $\text{SH}(k)^{\eta=0}_{\Lambda}$ denote the full subcategory on which $\eta$ acts trivially. On $\text{SH}(k)_{\Lambda}^+$ we have $h=2$, and the relation $\eta h=0$ gives $\eta=0$. Conversely, if $\eta=0$, then $\tau=1$. Hence
\begin{equation*}
    \text{SH}(k)_{\Lambda}^+\simeq\text{SH}(k)^{\eta=0}_{\Lambda}.
\end{equation*}

On $\text{SH}(k)_{\Lambda}^-$ the identity $\tau=-1$ gives $\eta[-1]=-2$, so $\eta$ acts invertibly. Conversely, on an $\eta$-local object the relation $\eta h=0$ implies $h=0$ and hence $\tau=-1$. With the localization notation introduced in Section \ref{prem}, this proves
\begin{equation*}
    \text{SH}(k)_{\Lambda}^-\simeq\text{SH}(k)_{\Lambda}[\eta^{-1}].
\end{equation*} 

For the plus part, Cisinski--D\'eglise \cite[Section 16.2]{Cisinski_2019} have described an equivalence with the big derived category of motives over $k$ with $\Q$-coefficients when $\Lambda=\Q$:
\begin{equation*}
    \text{SH}(k)_{\Q}^+\simeq\text{DM}(k)_{\Q}
\end{equation*}
with the inclusion $\text{SH}_{\Q}^+(k)\subset\text{SH}(k)_{\Q}$ corresponding to the functor $U: \text{DM}(k)\to\text{SH}(k)$ we mentioned in Section \ref{prem}.

For the minus part, Bachmann \cite[Theorem 35]{Bachmann_20181} has shown that when $k=\R$, we have
$\text{SH}(\mathbb{R})[\rho^{-1}]\simeq\text{SH}$, where $\rho=[-1]\in K^{MW}_1(k)$, and the equivalence is induced by the real realization functor. From the discussion above,
we know that $\rho$ corresponds to the map of pointed motivic spaces $S^0\to\mathbb{G}_m$ given by $-1\in\R$.

After inverting $2$, inverting $\eta$ is equivalent to inverting $\rho$, since on the minus part $-2=\eta[-1]$. Thus
\begin{equation*}
\text{SH}_{\Lambda}^-(\R)\simeq \text{SH}(\R)_{\Lambda}[\eta^{-1}]
\simeq\text{SH}(\mathbb{R})_{\Lambda}[\rho^{-1}]\simeq \text{SH}_{\Lambda}.
\end{equation*}
When $\Lambda=\Q$, we furthermore have $\text{SH}^-(\R)_{\Q}\simeq \text{SH}_{\Q}\simeq D(\Q)$.
The last equivalence via the singular chain complex functor $X\mapsto C_*(X(\R))\otimes\Q$ is a classical result in rational stable homotopy theory.

\begin{remark}\cite{Bachmann_20182}\label{eta}
    An important observation is that $M(\eta)=0$, so $M$ will annihilate $\text{SH}^{-}(k)$ and we can restrict $M$ to $\text{SH}_{\Lambda}^+(k)$.
\end{remark}

And to summarize this section, we have the following result:

\begin{proposition}\label{conclusion1}
    If $k=\R$ and $2\in\Lambda^\times$, the plus--minus decomposition and the two comparison functors give
    \begin{equation*}
        \text{SH}(\R)_{\Lambda}\simeq\text{SH}_{\Lambda}^+(\R)\times\text{SH}_{\Lambda}
        \longrightarrow\text{DM}(\R)_\Lambda\times\text{SH}_{\Lambda}.
    \end{equation*}
    If $\Lambda=\Q$, this is an equivalence of categories:
    \begin{equation*}
        \text{SH}(\R)_{\Q}\simeq\text{DM}(\R)_{\Q}\times\text{SH}_{\Q}.
    \end{equation*}
\end{proposition}

For a smooth $\R$-variety $X$, its associated motivic spectrum $\Sigma^{\infty,\infty}X_+$ will split into two parts in $\text{SH}(\R)_{\Lambda}$ with $2\in\Lambda^\times$.
If we work on $\Lambda=\Q$, one is given by the rational geometric motive of $X$ which will be introduced later, and the other one is given by the rational $S^1$-spectrum associated
to the space $X(\R)$ adjoining a point. 

So to understand the splitting further, we need to work on these two parts separately. Fortunately, they will both decompose into smaller 
pieces if $X$ is a real abelian variety.

\section{Decomposition of motives}\label{motive}
In this section we will review the construction of Chow motives and revisit the result of decomposition of diagonals on abelian varieties due to Deninger-Murre and finally get to the result by K\"unnemann.
And this decomposition gives us the splitting in $\text{SH}(k)_{\Lambda}^+$ when $\Lambda=\Q$. 

Let $X$ be an algebraic variety over a field $k$. The Chow group $\text{CH}_r(X)$ of $r$-dimensional cycles modulo rational equivalence is defined as the quotient 
group of the free abelian group $\oplus_V \Z [V]$, where the sum runs over all irreducible subvarieties $V\subset X$ of dimension $r$, by the subgroup generated by divisors of rational functions. That is, for any rational function $f$ on an $(r+1)$-dimensional subvariety $U$ of $X$, the divisor $\text{div}(f)$ is a formal
sum of $r$-dimensional subvarieties of $U$ corresponding to the zeros and poles of $f$.

If $X$ is equidimensional, one may also use the co-dimensional notation $\text{CH}^r(X)=\text{CH}_{\text{dim}(X)-r}(X)$.

The definition of Chow groups is an algebro-geometric analogue of (co)homology in topology. Chow groups form an \emph{oriented cohomology theory} on $SmVar_k$. That 
is, for any morphism $f:X\to Y$ of smooth varieties, there is a pullback map $f^*:\text{CH}^*(Y)\to\text{CH}^*(X)$, and for any proper morphism, a
push-forward map $f_*:\text{CH}_*(X)\to\text{CH}_*(Y)$, satisfying standard functorial properties. In particular, $\text{CH}^*(X)$ carries a ring structure given by intersection
product $\alpha\cdot\beta=\Delta^*(\alpha\times\beta)$, where $\Delta:X\to X\times X$ is the diagonal embedding.

Following Grothendieck and Manin, we now define the category of Chow correspondences $\text{Cor}_k$; see \cite{Manin1968} and \cite[Section 16.1]{fulton}. Its objects are smooth projective $k$-varieties and $\text{Hom}_{\text{Cor}(k)}([X],[Y])=\text{CH}_{\text{dim}(X)}(X\times Y)$,
assuming $X$ is equidimensional. Cycles in $\text{CH}_*(X\times Y)$ are called \emph{correspondences} from $X$ to $Y$. The composition $\circ$ of correspondences is defined by
\begin{equation*}
    \psi\circ\varphi=(\pi_{X,Z})_*((\pi_{X,Y})^*\varphi\cdot(\pi_{Y,Z})^*\psi),
\end{equation*}
where $\varphi\in\text{CH}_{\text{dim}(X)}(X\times Y),\psi\in\text{CH}_{\text{dim}(Y)}(Y\times Z)$ and $\pi_{X,Y},\pi_{Y,Z},\pi_{X,Z}$ are the natural projections from $X\times Y\times Z$ to $X\times Y$, $Y\times Z$ and $X\times Z$ respectively.
The composition is associative, and if we compose this way the graphs of morphisms, it will coincide with the composition of morphisms themselves. Thus, we get the functor 
\begin{equation*}
    \text{Cor}:\text{SmProj}_k\to\text{Cor}(k)
\end{equation*}
which sends $X$ to itself and $(f:X\to Y)$ to its graph $[\Gamma_f]\in\text{CH}_{\text{dim(X)}}(X\times Y)$.

The category $\text{Cor}(k)$ admits direct sums $X\oplus Y=X\coprod Y$. It also has the natural tensor structure given by $X\otimes Y=X\times Y$. 

We define the category of \emph{effective Chow motives}, $\text{Chow}_{eff}(k)$, as the pseudo-abelian (Karoubian) envelope of $\text{Cor}(k)$. Its objects are pairs $(X,\rho)$, where 
$X$ is a smooth projective variety and $\rho\in\text{End}_{\text{Cor}(k)}(X)$ is a projector (i.e. $\rho\circ\rho=\rho$). Morphisms between objects are given by $\text{Hom}_{\text{Chow}_{eff}}((X,\rho),(Y,\eta))=\eta\circ\text{Hom}_{\text{Cor}(k)}(X,Y)\circ\rho\subset\text{Hom}_{\text{Cor}(k)}(X,Y)$.
In other words, we formally add kernels and cokernels of projectors. There is a tensor product on $\text{Chow}_{eff}(k)$ defined by:
\begin{equation*}
    (X,\rho)\otimes(Y,\eta)=(X\times_k Y, \rho\times_k \eta).
\end{equation*}

We also define the \emph{Chow motivic functor}:
\begin{equation*}
    Chow:\text{SmProj}_k\to\text{Chow}_{eff}(k),
\end{equation*}
which maps a smooth projective variety $X$ to the pair $(X,[\Delta_X])$, where $[\Delta_X]$ is the class of the diagonal.

In the category $\text{Chow}_{eff}(k)$, the motive of certain varieties splits into simpler components.

For example, let $X=\mathbb{P}^1$. Using the rational function $\frac{x_0 y_1-x_1 y_0}{x_0 y_0}$, we find that the class of the diagonal is rationally equivalent to the sum $[pt \times\mathbb{P}^1]+[\mathbb{P}^1\times pt]$, corresponding to
mutually orthogonal projectors. Considering the maps $pt \stackrel{f}{\to}\mathbb{P}^1\stackrel{g}{\to}pt$, the second projector gives a direct summand in $Chow(\mathbb{P}^1)$
isomorphic to $Chow(pt):=1$, called the \emph{trivial Tate motive}. The complementary summand given by the projector $[pt\times\mathbb{P}^1]$ is denoted
by $L$, called the \emph{Lefschetz motive}. Hence, $Chow(\mathbb{P}^1)=1 \oplus L$. 

The category $\text{Chow}(k)$ of \emph{Chow motives} is obtained from $\text{Chow}_{eff}(k)$ by formally inverting $L$ under the tensor product. We define $L=\Z(1)$, which is the \emph{Tate motive}. Thus, to define objects in $\text{Chow}(k)$, 
we allow an additional integer $n\in\Z$ to track the Tate motives. The effective category $\text{Chow}_{eff}(k)$ is recovered as the full subcategory of $\text{Chow}(k)$ where $n=0$. The tensor product on $\text{Chow}(k)$ is given by:
\begin{equation*}
    (X,\rho,m)\otimes(Y,\eta,n)=(X\times_k Y, \rho\times_k \eta, m+n),
\end{equation*}
so the category of Chow motives is tensor additive. We denote $\Z(n)=\Z(1)^{\otimes n}$, and define Tate twisting in $\text{Chow}(k)$ by $(X,p,m)\otimes\Z(n):=(X,p,m)(n)=(X,p,m+n)$.

The category of Chow motives can be embedded as a full tensor additive subcategory into the Voevodsky category of motives $\text{DM}_{gm}(k)$. Detailed treatments of this embedding can be found in \cite[Proposition 2.1.4, Theorem 3.2.6]{59}. Under this embedding, our Chow-motivic $\Z(n)$ corresponds to $\Z(n)[2n]$ in Voevodsky's notation. In conjunction with Proposition \ref{tensor}, we observe that for a smooth projective variety $X$, $M(\Sigma^{\infty,\infty} X_+)$ 
actually lies within the category of Chow motives. In particular, since an abelian variety is both smooth and projective \cite{mumford1974abelian}, it is natural to study the decomposition of its Chow motive. For clarity and without risk of confusion, we will denote $Chow(X)=M_{gm}(X)$ in what follows.

Extending scalars in the Chow groups from $\Z$ to $\Q$ gives the category of rational Chow motives, which embeds as a full subcategory of $\text{DM}(k)_{\Q}$. The same construction gives Chow motives with $\Lambda$-coefficients for any coefficient ring $\Z\subset\Lambda\subset\Q$.
For simplicity, we use the same notation for rational Chow motives and for Chow motives with $\Lambda$-coefficients.

\begin{remark}\label{twist}
    Proposition \ref{tensor} and the constructions of $\text{SH}(k)$ and $\text{DM}(k)$ give $M(S^{0,0})=1$ and $M(S^{2,1})\cong L=\Z(1)$. Since $M$ is tensor exact, smashing with $S^{2,1}$ corresponds under motivization to tensoring with the Lefschetz motive $\Z(1)$.
\end{remark}

We now consider an abelian variety $X$ of dimension $g\geq1$ over a field $k$ of characteristic zero. Put
\[
    \Lambda_g=\Z\left[\frac{1}{(2g)!}\right].
\]
The same discussion applies over $\Z[1/(2g+1)!]$ under the more general hypotheses of \cite[Section 3.2]{HHLM2025}. The following form of the Deninger--Murre decomposition over $\Lambda_g$ is the one needed below.

\begin{theorem}\label{decomposition}
    There is a unique decomposition in $\text{CH}^{g}(X\times_k X;\Lambda_g)$
    \begin{equation*}
        [\Delta_X]=\sum_{i=0}^{2g}\pi_i
    \end{equation*}
    into pairwise orthogonal idempotents such that
    \begin{equation*}
        [\Gamma_{[n]_X}]\circ\pi_i=n^{2g-i}\pi_i
        \qquad (n\in\Z).
    \end{equation*}
    Moreover, every $\pi_i$ is a finite $\Lambda_g$-linear combination of graph correspondences $[\Gamma_{[n]_X}]$.
\end{theorem}
\begin{proof}
    Let
    \[
       C=\text{CH}^{g}(X\times_kX;\Lambda_g),\qquad
       \gamma_n=[\Gamma_{[n]_X}],\qquad m=2g.
    \]
    Besides composition of correspondences, $C$ has the relative Pontryagin product $*$ for the abelian scheme $X\times_kX\to X$ given by the first projection. Thus
    \begin{equation}\label{twoproducts}
       \gamma_a*\gamma_b=\gamma_{a+b},\qquad
       \gamma_a\circ\gamma_b=\gamma_{ab},\qquad
       \gamma_0=1_*,\qquad \gamma_1=[\Delta_X].
    \end{equation}
    We first prove the integral relation
    \[
       (g-1)!\,(\gamma_1-\gamma_0)^{*(m+1)}=0
       \quad\text{in }\text{CH}^g(X\times_kX).
    \]
    Consider the relative Fourier--Mukai equivalence for the constant abelian scheme $X\times_kX\to X$ and its dual $X\times_k\widehat X\to X$, with normalized Poincar\'e bundle $P$ on $X\times_k\widehat X$; see \cite[Proposition 1.2.3 and Theorem 1.2.4]{Laumon1996}. Its induced map on $K_0$ satisfies
    \[
       \Phi_K([\mathcal O_{\Gamma_{[n]}}])=[P^{\otimes n}].
    \]

    Put $Y=X\times_k\widehat X$. Multiplication by $[P]-1$ lowers the dimension-of-support filtration on $G_0(Y)$. As $Y$ is regular of dimension $2g$, $K_0(Y)=G_0(Y)$ and
    \[
       ([P]-1)^{m+1}=0\quad\text{in }K_0(Y);
    \]
    see also \cite[Chapter II, Remark 8.8.5]{Weibel2013}. Applying the inverse of $\Phi_K$ to the binomial expansion gives
    \[
       \sum_{n=0}^{m+1}(-1)^{m+1-n}\binom{m+1}{n}
       [\mathcal O_{\Gamma_{[n]}}]=0\quad\text{in }K_0(X\times_kX).
    \]

    Each graph is a regular closed subvariety of codimension $g$. Its Chern classes of positive degree below $g$ vanish by localization, and Riemann--Roch without denominators gives
    \[
       c_g(\mathcal O_{\Gamma_{[n]}})
       =(-1)^{g-1}(g-1)!\,\gamma_n;
    \]
    see \cite[Section 1, formula (1)]{PodkopaevShinder2006}. The total Chern class is multiplicative, so its degree-$g$ term is additive on integral virtual combinations of these graph classes, including negative coefficients. Applying $c_g$ to the preceding $K_0$ identity proves the integral relation, since its binomial sum is $(\gamma_1-\gamma_0)^{*(m+1)}$. Inverting $(g-1)!$ gives
    \begin{equation}\label{strictnilpotence}
       (\gamma_1-\gamma_0)^{*(m+1)}=0
       \qquad\text{in }C.
    \end{equation}
    Therefore the following logarithm and exponential are finite polynomial expressions. Set
    \begin{equation*}
       t=\log_*(\gamma_1)
        =\sum_{r=1}^{m}(-1)^{r-1}\frac{(\gamma_1-\gamma_0)^{*r}}{r},
       \qquad q_j=\frac{t^{*j}}{j!}\quad(0\leq j\leq m).
    \end{equation*}
    Formula \eqref{twoproducts}, including the fact that $\gamma_{-1}$ is the $*$-inverse of $\gamma_1$, gives for every $n\in\Z$
    \begin{equation}\label{graphexpansion}
       \gamma_n=\exp_*(nt)=\sum_{j=0}^{m}n^j q_j.
    \end{equation}
    In particular, taking $n=1$ yields $\sum_jq_j=\gamma_1=[\Delta_X]$.

    It remains important to prove orthogonality as an identity of correspondences, rather than merely after their action on Chow groups. For $1\leq a,b\leq m+1$, equations \eqref{twoproducts} and \eqref{graphexpansion} give
    \begin{equation*}
       \sum_{j,k=0}^{m}a^jb^k(q_j\circ q_k)
       =\gamma_a\circ\gamma_b
       =\gamma_{ab}
       =\sum_{r=0}^{m}a^rb^r q_r.
    \end{equation*}
    The Vandermonde matrix $V=(a^j)_{1\leq a\leq m+1,\,0\leq j\leq m}$ has determinant
    \begin{equation*}
       \det(V)=\prod_{1\leq a<b\leq m+1}(b-a)=\prod_{r=1}^{m}r!,
    \end{equation*}
    which is a unit in $\Lambda_g$. Multiplication by $V^{-1}$ and $(V^{\mathsf t})^{-1}$ therefore gives
    \begin{equation*}
       q_j\circ q_k=\delta_{jk}q_j.
    \end{equation*}
    Define $\pi_i=q_{2g-i}$. Orthogonality and \eqref{graphexpansion} imply
    $\gamma_n\circ\pi_i=n^{2g-i}\pi_i$. Expanding the powers of $\gamma_1-\gamma_0$ and using $\gamma_a*\gamma_b=\gamma_{a+b}$ shows that $t$, and hence every $q_j$, is a finite $\Lambda_g$-linear combination of the $\gamma_n$.

    Finally, suppose $\{\pi_i'\}$ is another family with the stated properties. Then
    \begin{equation*}
       \gamma_n=\sum_{i=0}^{2g}n^{2g-i}\pi_i
               =\sum_{i=0}^{2g}n^{2g-i}\pi_i'
    \end{equation*}
    for every $n$. Evaluating at $n=0,\ldots,2g$ gives a Vandermonde system whose determinant is $\prod_{r=1}^{2g}r!\in\Lambda_g^\times$; hence $\pi_i=\pi_i'$ for all $i$. After extension of scalars to $\Q$, this is the classical decomposition of \cite[Theorem 3.1]{Deninger1991}.
\end{proof}

Moreover, this decomposition also induces a corresponding decomposition at the level of motives.

\begin{theorem}
    Under the hypotheses of Theorem \ref{decomposition}, define
    \begin{equation*}
        M_{gm}^i(X):=(X,\pi_{2g-i},0),
    \end{equation*}
    with $\pi_i$ as in Theorem \ref{decomposition}. Then Theorem \ref{decomposition} yields a canonical decomposition
    \begin{equation*}
        M_{gm}(X)\simeq\bigoplus_{i=0}^{2g} M_{gm}^i(X)
    \end{equation*}
    in $\text{Chow}(k;\Lambda_g)$ and hence in $\text{DM}(k)_{\Lambda_g}$.
\end{theorem}
\begin{proof}
    This is a direct result from Theorem \ref{decomposition}.
\end{proof}

We next recall K\"unnemann's finer decomposition with rational coefficients. Suppose $X$ is polarized by an ample divisor class $d$ represented by a symmetric line bundle, meaning that $[-1]^*\mathcal O_X(d)\cong\mathcal O_X(d)$. Let $L_d$ be the Lefschetz operator given by multiplication by the class of $d$.
Then we have the following result due to K\"unnemann \cite{Kunnemann1993}:

\begin{theorem}\label{Kunn}
    For $i\in\{0,\ldots,2g\}$, the rational Chow motive $M_{gm}^i(X)$ has a Lefschetz decomposition
    \begin{equation*}
        M_{gm}^i(X)\simeq\bigoplus^{\lfloor i/2\rfloor}_{k=\max\{0,i-g\}}P^{i-2k}(X)(k)
    \end{equation*}
    in $\text{Chow}(k)_{\Q}$, with the Lefschetz morphisms and primitive pieces described in \cite[Section 4]{Kunnemann1993}.
\end{theorem}

Here $P^i(X)(k)$ is a direct factor of $M^{i+2k}_{gm}(X)$. In particular, $P^i(X)=L_d^0 P^i(X)$ denotes the primitive part of $M^i_{gm}(X)$. These $P^i(X)$'s are determined recursively by $M^i_{gm}(X)$. For example,
one can check from the above theorem that $P^0(X)=M^0_{gm}(X)$ and $P^1(X)=M^1_{gm}(X)$. Everything in the decomposition is determined by $M^1_{gm}(X)$. This follows from the recursive construction of $P^i(X)$ \cite{Kunnemann1993}.
This also reflects the fact that the Weil cohomology ring of $X$ over any characteristic 0 field is determined by the first cohomology group and we can also lift this to Chow motives.

From Theorem \ref{Kunn} we also have 

\begin{theorem}
    For $i\in\{0,...,g\}$,
    \begin{equation*}
        L^{g-i}_d: M_{gm}^{2g-i}(X)\stackrel{\sim}{\to} M_{gm}^i(X)(g-i)
    \end{equation*}
    is an isomorphism in $\text{Chow}(k)_{\Q}$.
\end{theorem}

Combining the preceding theorems gives the following rational decomposition:

\begin{theorem}\label{Kunn1}
    Let $X$ be an abelian variety over $k$ of dimension $g$. Then 
    \begin{equation*}
        M_{gm}(X)\simeq\bigoplus_{i=0}^{g-1}\bigoplus_{k=0}^{\lfloor\frac{i}{2}\rfloor} (P^{i-2k}(X)(k)\oplus P^{i-2k}(X)(k+g-i) )\oplus\bigoplus_{k=0}^{\lfloor\frac{g}{2}\rfloor} P^{g-2k}(X)(k)
    \end{equation*}
    in $\text{Chow}(k)_{\Q}$.
\end{theorem}

Under the rational equivalence $\text{SH}(k)^+_{\Q}\simeq\text{DM}(k)_{\Q}$, this gives the rational Lefschetz--primitive refinement of the plus part used in Theorem \ref{abeliansplitting}.

\begin{remark}\label{projectivespace}
    Another example of the decomposition of the Chow motives is the case when $X=\mathbb{P}^m$. One can see that $\rho_i=[\mathbb{P}^{m-i}\times \mathbb{P}^i]$ are mutually orthogonal
    projectors on $\mathbb{P}^m$ and $\sum_{i=0}^m\rho_i=[\Delta]$ and $(\mathbb{P}^m,\rho_i)\cong L^{\otimes i}$. Thus, $Chow(\mathbb{P}^m)=\bigoplus_{i=0}^m L^{\otimes i}$. This means that $(\mathbb{P}^m_+)^+$ will
    split into pieces in $\text{SH}(k)^+_{\Q}$. This contrasts with Remark \ref{projective} because the stable algebraic Hopf map $\eta$ vanishes in $\text{SH}(k)^+$.
\end{remark}

\begin{remark}\label{conservative}
    The coefficient ring $\Lambda_g=\Z[1/(2g)!]$ suffices for the argument. Deducing \eqref{strictnilpotence} from the integral relation only requires the inversion of $(g-1)!$, and the logarithm, exponential and Vandermonde calculations involve only prime divisors of $(2g)!$. We make no optimality claim for the coefficient ring.

    The nilpotence argument below lifts any motivic projector that is already in the image of motivization. The Deninger--Murre projectors satisfy this image condition because they are linear combinations of multiplication graphs. K\"unnemann's Lefschetz--primitive projectors involve the polarization Lefschetz operator and its Fourier correspondence; they are not known over $\Lambda_g$ to be combinations of multiplication graphs. Hence no lift of K\"unnemann's refinement over $\Lambda_g$ is asserted here. Rationally there is no image obstruction, because $\text{SH}(k)^+_{\Q}\simeq\text{DM}(k)_{\Q}$.
\end{remark}

We now prove the lifting statement over $\Lambda_g$. Let
\[
 E_X=(\Sigma_T^\infty X_+)^+_{\Lambda_g},\qquad
 \rho_X:\operatorname{End}_{\text{SH}(k)^+_{\Lambda_g}}(E_X)
 \longrightarrow
 \operatorname{End}_{\text{DM}(k)_{\Lambda_g}}(M_{gm}(X))
\]
be induced by motivization.

\begin{proposition}\label{nilpotencekernel}
Let $X/k$ be a smooth projective variety of dimension $g\geq1$ over a perfect field. Then
\[
   \ker(\rho_X)^{g+1}=0.
\]
\end{proposition}
\begin{proof}
Write $M:\text{SH}(k)^+_{\Lambda_g}\rightleftarrows
\text{DM}(k)_{\Lambda_g}:U$ for motivization and its right adjoint. Let
\[
 B_0=h^{\text{DM}}_0(\mathbb I_{\text{DM}}),\qquad
 \mathcal H_0=U(B_0),\qquad
 K=\operatorname{fib}(\mathbb I^+_{\Lambda_g}\longrightarrow\mathcal H_0).
\]
Here the last map is the adjunction unit followed by the zeroth homotopy truncation of the motivic unit. It is essential that $\mathcal H_0$ is the underlying spectrum of the heart object $B_0$, rather than the full spectrum $U(\mathbb I_{\text{DM}})$.
Both $\mathbb I^+_{\Lambda_g}$ and $E_X$ are connective: suspension spectra of smooth schemes are connective, and coefficient localization and the plus projection preserve connectivity.

The functor $U$ is $t$-exact for the homotopy $t$-structures \cite[Lemma 7]{Bachmann_20182}, so $\mathcal H_0$ lies in the heart. Bachmann's motivic Hurewicz theorem gives, for a connective spectrum $E$,
\[
 h^{\text{DM}}_0(ME)_*\cong\underline\pi_0(E)_*/\eta;
\]
see \cite[Theorem 12]{Bachmann_20182}, whose identification of the hearts uses \cite[Theorem 1.3.4]{Deglise2013}. Applying this to $E=\mathbb I^+_{\Lambda_g}$, on which $\eta=0$, shows that
\[
 \underline\pi_0(\mathbb I^+_{\Lambda_g})_*
 \xrightarrow{\ \sim\ }
 \underline\pi_0(\mathcal H_0)_*.
\]
Since $\mathcal H_0$ is $0$-truncated, the long exact homotopy sequence yields
\begin{equation}\label{Kconnective}
 K\in\text{SH}(k)_{\geq1}.
\end{equation}

The object $E_X$ is strongly dualizable because $X$ is smooth and proper \cite{Riou2005}. For every $T\in\text{SH}(k)^+_{\Lambda_g}$, duality and the adjunction $M\dashv U$ give natural isomorphisms
\begin{align*}
 \operatorname{Hom}(T,\mathcal H_0\wedge E_X)
 &\cong\operatorname{Hom}(T\wedge E_X^\vee,\mathcal H_0)\\
 &\cong\operatorname{Hom}(M(T)\wedge M(E_X)^\vee,B_0)\\
 &\cong\operatorname{Hom}(M(T),B_0\wedge M(E_X))\\
 &\cong\operatorname{Hom}(T,U(B_0\wedge M(E_X))).
\end{align*}
The Yoneda lemma therefore gives the projection-formula equivalence
\begin{equation}\label{projectionH0}
 \mathcal H_0\wedge E_X\simeq U(B_0\wedge M(E_X)).
\end{equation}
Consequently, if $f\in\ker(\rho_X)$, then $\mathcal H_0\wedge f=0$. Naturality of the unit implies that
\[
 E_X\xrightarrow{f}E_X\longrightarrow\mathcal H_0\wedge E_X
\]
is zero, and the fibre sequence obtained by smashing the definition of $K$ with $E_X$ shows that $f$ factors through $K\wedge E_X$.

Iterating this argument, a product $f_r\cdots f_1$ of $r$ elements of $\ker(\rho_X)$ factors through $K^{\wedge r}\wedge E_X$. Indeed, if a product factors through $K^{\wedge r}\wedge E_X\to E_X$, naturality moves the next endomorphism across this last map as $K^{\wedge r}\wedge f$; equation \eqref{projectionH0} then factors that map through $K^{\wedge(r+1)}\wedge E_X$. By \eqref{Kconnective} and compatibility of the homotopy $t$-structure with smash products, $K^{\wedge(g+1)}\wedge E_X$ is $(g+1)$-connective.

Finally, if $Z$ is a bounded-below motivic spectrum, the Brown--Gersten spectral sequence has the form
\[
 E_2^{p,q}=H^p_{\mathrm{Nis}}(X,\underline\pi_q(Z)_0)
 \Longrightarrow [\Sigma_T^\infty X_+,Z[p-q]];
\]
see \cite[Lemma 3.3.3]{MorelStable2005} for the corresponding dimension vanishing. For a $(g+1)$-connective $Z$, every term contributing to total degree zero vanishes: if $p=q\leq g$, this follows from connectivity, and if $p=q>g$, it follows from $\operatorname{cd}_{\mathrm{Nis}}(X)\leq g$. Hence
\[
 [E_X,K^{\wedge(g+1)}\wedge E_X]=0,
\]
because $E_X$ is a retract of the localized suspension spectrum of $X_+$. Every product of $g+1$ elements of $\ker(\rho_X)$ is therefore zero.
\end{proof}

\begin{theorem}\label{DMlift}
Let $X/k$ be an abelian variety of dimension $g\geq1$ in characteristic zero, and let $\Lambda$ be any subring of $\Q$ containing $\Lambda_g$. There are pairwise orthogonal idempotents
\[
 \widetilde\pi_0,\ldots,\widetilde\pi_{2g}
 \in\operatorname{End}_{\text{SH}(k)^+_\Lambda}
       ((\Sigma_T^\infty X_+)^+_\Lambda)
\]
whose motivizations are the Deninger--Murre projectors and whose sum is the identity. Let $\mathbb H^i(X)$ denote the retract of $(\Sigma_T^\infty X_+)^+_\Lambda$ determined by $\widetilde\pi_{2g-i}$. Then there is a decomposition
\begin{equation}\label{DMplussplitting}
 (\Sigma_T^\infty X_+)^+_\Lambda
 \simeq\bigvee_{i=0}^{2g}\mathbb H^i(X),
 \qquad M(\mathbb H^i(X))\simeq M^i_{gm}(X).
\end{equation}
The same assertion holds over a perfect field whenever $X$ satisfies the more general hypotheses of \cite[Section 3.2]{HHLM2025} and $\Z[1/(2g+1)!]\subset\Lambda$.
\end{theorem}
\begin{proof}
In characteristic zero, it is enough to work over $\Lambda_g$ and then extend scalars; under the more general hypotheses above, use $\Z[1/(2g+1)!]$ in place of $\Lambda_g$ throughout the argument. Let $R_X$ be the unital commutative $\Lambda_g$-subalgebra of $\operatorname{End}(E_X)$ generated by the stable multiplication maps $[n]_{X,+}$. The algebra is commutative because $[a]_X\circ[b]_X=[ab]_X$. Put
\[
 I_X=R_X\cap\ker(\rho_X).
\]
Proposition \ref{nilpotencekernel} gives $I_X^{g+1}=0$.

By construction of motivization, the stable map $[n]_{X,+}$ is sent to the graph correspondence $\gamma_n=[\Gamma_{[n]_X}]$. Theorem \ref{decomposition} expresses every $\pi_i$ as a finite $\Lambda_g$-linear combination of these graph correspondences. Hence the complete orthogonal family $(\pi_i)$ lies in the image of $R_X$, and defines a complete orthogonal family of idempotents in $R_X/I_X$.

Idempotents lift uniquely across a nilpotent ideal in a commutative ring \cite[Lemma 10.32.6, Tag 00J9]{Stacks00J9}. Let $\widetilde\pi_i\in R_X$ be the lift of $\pi_i$. If $i\neq j$, then $\widetilde\pi_i\widetilde\pi_j$ is an idempotent lifting zero and is therefore zero. Their sum is an idempotent lifting one and hence equals one. Since $\text{SH}(k)^+_{\Lambda_g}$ is idempotent complete, these idempotents split $E_X$ and give \eqref{DMplussplitting}.
\end{proof}

\section{A presentation of the splitting over the real numbers}\label{final}
Section \ref{motive} determines the plus part. It remains to determine the splitting in $\text{SH}^-(\R)$. Appendix B determines both the number and the topology of the connected components of $X(\R)$; each component is a product of copies of $S^1$. Classical stable splitting results then determine the minus part.

First recall the stable splitting of the product of CW complexes after taking reduced suspension once:

\begin{proposition}\cite{MR1867354}\label{stablesplitting}
    If $X$ and $Y$ are CW complexes, then $\Sigma (X\times Y)\simeq\Sigma X\vee\Sigma Y\vee\Sigma(X\wedge Y)$.
\end{proposition}
\begin{proof}
    See \cite[Proposition 4I.1]{MR1867354}.
\end{proof}

For a real abelian variety $X$ of dimension $g$, from Proposition \ref{realabelian1}, we know that each component of $X(\R)$ is $(S^1)^g$. So for each component, we have the stable
splitting:

\begin{lemma}\label{topocomponent}
    $\Sigma (S^1)^g\simeq \bigvee^{g}_{i=1}\bigvee^{\binom{g}{i}} S^{i+1}$, and in total this splitting has $\sum_{i=1}^g \binom{g}{i}=2^g-1$ summands.
\end{lemma}
\begin{proof}
    We use induction on $g\geq 2$. The base case $g=2$ is the familiar stable splitting of the torus $S^1\times S^1$. Here the reduced suspension $\Sigma (S^1\times S^1)$ is homotopy equivalent to
    $S^2\vee S^2\vee S^3$ since $\Sigma(S^1\times S^1)$ is $S^2\vee S^2$ with a 3-cell attached by the suspension of the attaching map of the 2-cell of the torus, but the latter attaching 
    map is the commutator of the two inclusions $S^1\hookrightarrow S^1\vee S^1$, and the suspension of this commutator is trivial since it lies in the abelian group $\pi_2(S^2\vee S^2)$. With this first step established, assume the claim holds for $g=k$. Proposition \ref{stablesplitting} then gives
    \[
    \Sigma(S^1)^{k+1}\simeq
    \bigvee^{k}_{i=1}\bigvee^{\binom{k}{i}} S^{i+1}
    \vee S^2\vee
    \Sigma\!\left(\bigvee^{k}_{i=1}\bigvee^{\binom{k}{i}} S^{i+1}\right).
    \]
    Distributing suspension over the wedge and using
    $\binom{k+1}{i}=\binom{k}{i}+\binom{k}{i-1}$ gives
    $\Sigma(S^1)^{k+1}\simeq\bigvee^{k+1}_{i=1}\bigvee^{\binom{k+1}{i}} S^{i+1}$, completing the induction.
\end{proof}

Next, taking suspension will connect these connected components together.

\begin{lemma}\label{topocomponent2}
    If $X$ is a real abelian variety of dimension $g$ and $X(\R)$ has $n(X)=2^d>0$ connected components, then
    $\Sigma (X(\R)_+)\simeq \bigvee^{n(X)}S^1\vee\bigvee^{n(X)}\Sigma (S^1)^g$.
    This splitting has $n(X)\cdot2^g$ summands by Lemma \ref{topocomponent}.
\end{lemma}
\begin{proof}
    By Proposition \ref{realabelian1} (iii), we know that $X(\R)$ can be viewed as $(S^1)^g\times (S^0)^d$. We again induct on $d$. For $d=0$, $\Sigma(((S^1)^g)_+)\simeq S^1\vee\Sigma(S^1)^g$. For the base case $d=1$, Proposition \ref{stablesplitting}
    tells us that if $X(\R)=(S^1)^g\times S^0$, then $\Sigma X(\R)_+\simeq S^1\vee\Sigma (S^1)^g\vee S^1\vee \Sigma (S^1)^g$. And if we assume the claim holds for $d=k$, then we have $\Sigma ((S^1)^g\times(S^0)^d)_+\simeq\vee^{2^d}S^1\bigvee^{2^d}\Sigma (S^1)^g$.
    Then again by Proposition \ref{stablesplitting}, $\Sigma((S^1)^g\times(S^0)^{d+1})_+\simeq S^1\vee(\vee^{2^d-1}S^1\bigvee^{2^d}\Sigma (S^1)^g)\vee S^1\vee(\vee^{2^d-1}S^1\bigvee^{2^d}\Sigma (S^1)^g)\simeq\vee^{2^{d+1}}S^1\bigvee^{2^{d+1}}\Sigma (S^1)^g$.
    So by induction we conclude the proof.
\end{proof}

The preceding decompositions occur after one suspension and hence describe the stable splitting of $X(\R)_+$ in the classical stable homotopy category. They persist after passage to any coefficient ring $\Lambda$. Over $\Lambda_g$, a motivic decomposition lifts only after both the image condition and the nilpotence condition have been verified; these are precisely Theorem \ref{decomposition} and Proposition \ref{nilpotencekernel}. Rationally, the equivalence $\text{SH}(k)^+_{\Q}\simeq\text{DM}(k)_{\Q}$ also gives K\"unnemann's refinement.

With the preceding results, we can now formulate the main theorem.

\begin{theorem}\label{mainreal}
Let $X$ be a real abelian variety of dimension $g\geq1$, and let $\Z[1/(2g)!]\subset\Lambda\subset\Q$. Then there are objects $\mathbb H^r(X)\in\text{SH}(\R)^+_\Lambda$, for $0\leq r\leq2g$, satisfying $M(\mathbb H^r(X))\simeq M^r_{gm}(X)$, and
\begin{equation}\label{realDMsplitting}
 X_+\simeq
 \left(\bigvee_{r=0}^{2g}\mathbb H^r(X)\right)
 \vee
 \left(\bigvee^{n(X)}\bigvee_{i=0}^{g}
             \bigvee^{\binom gi}S^{i,0,-}_\Lambda\right)
 \quad\text{in }\text{SH}(\R)_\Lambda.
\end{equation}
Here $n(X)$ is the number of connected components of $X(\R)$; under the additional hypotheses of Theorem \ref{components1}, it is given explicitly by Huisman's formula there.

If $\Lambda=\Q$ and a polarization is chosen, the plus term in \eqref{realDMsplitting} further refines as
\begin{equation}\label{rationalKunnsplitting}
 \bigvee_{i=0}^{g-1}\left(\bigvee_{k=0}^{\lfloor i/2\rfloor}
 (S^{2k,k}_{\Q}\vee S^{2(k+g-i),k+g-i}_{\Q})
 \wedge\mathbb J_{i-2k}(X)\right)
 \vee\bigvee_{k=0}^{\lfloor g/2\rfloor}
 S^{2k,k}_{\Q}\wedge\mathbb J_{g-2k}(X),
\end{equation}
where $\mathbb J_i(X)\in\text{SH}(\R)^+_{\Q}$ and $M(\mathbb J_i(X))\simeq P^i(X)$.
\end{theorem}
\begin{proof}
Theorem \ref{DMlift} gives the plus decomposition
\[
 (\Sigma_T^\infty X_+)^+_\Lambda
 \simeq\bigvee_{r=0}^{2g}\mathbb H^r(X),
 \qquad M(\mathbb H^r(X))\simeq M^r_{gm}(X).
\]
For the minus part, Bachmann's equivalence
$\text{SH}(\R)^-_\Lambda\simeq\text{SH}_\Lambda$ identifies the minus summand of $X_+$ with the classical suspension spectrum of $X(\R)_+$. Lemma \ref{topocomponent2} and the stable splitting of $(S^1)^g$ give
\[
 (\Sigma_T^\infty X_+)^-_\Lambda
 \simeq\bigvee^{n(X)}\bigvee_{i=0}^{g}
             \bigvee^{\binom gi}S^{i,0,-}_\Lambda.
\]
Combining the plus and minus splittings proves \eqref{realDMsplitting}.

When $\Lambda=\Q$, the equivalence $\text{SH}(\R)^+_{\Q}\simeq\text{DM}(\R)_{\Q}$ transports K\"unnemann's decomposition in Theorem \ref{Kunn1} to the plus part; Remark \ref{twist} identifies a Tate twist by $\Z(1)$ with smashing by $S^{2,1}$. This gives \eqref{rationalKunnsplitting}. Since every direct summand of the motive of an abelian variety is of abelian type, a suitable Tate twist is a direct summand of the motive of a product of curves \cite[p.~1]{abeliantype}.
\end{proof}

\begin{remark}\label{chowwitt1}
The splitting also has a direct Chow--Witt interpretation. Let
\[
 M_{\mathrm{MW},\Lambda}\colon
 \mathrm{SH}(\R)_\Lambda\longrightarrow
 \widetilde{\mathrm{DM}}(\R,\Lambda)
\]
denote the exact symmetric monoidal Milnor--Witt motivization functor
\cite[Chapter~7, Section~5]{bachmann2022milnorwitt}. Applying this functor
to \eqref{realDMsplitting} gives a corresponding decomposition of the
Milnor--Witt motive
\[
 \widetilde M_\Lambda(X):=
 M_{\mathrm{MW},\Lambda}(\Sigma_T^\infty X_+).
\]
Since $2\in\Lambda^\times$, there is an equivalence
\[
 \widetilde{\mathrm{DM}}(\R,\Lambda)
 \simeq
 \mathrm{WDM}(\R,\Lambda)\times\mathrm{DM}(\R,\Lambda)
\]
by \cite[Chapter~3, Theorem~5.0.2]{bachmann2022milnorwitt}. The same
product can be seen directly on Chow--Witt groups. Morel's Cartesian
square for Milnor--Witt $K$-theory gives, after tensoring with $\Lambda$,
\[
 \mathbf K_q^{\mathrm{MW}}(\mathcal L)_\Lambda
 \simeq
 \mathbf K_q^{\mathrm M}{}_{\Lambda}
 \times\mathbf I^q(\mathcal L)_\Lambda,
\]
because the common quotient in the Cartesian square is annihilated by
$2$; see \cite[Theorem~5.3]{Morel04}. For a trivial twist, the inclusion
$\mathbf I^q\subset\mathbf W$ becomes an isomorphism after inverting $2$,
since $2^q\mathbf W\subset\mathbf I^q$.
Jacobson's signature comparison
\cite[Theorem~8.6 and Corollary~8.9]{10.2140/akt.2017.2.357} therefore
gives, for every smooth real variety $Y$,
\begin{equation}\label{localizedCWproduct}
 \widetilde{\mathrm{CH}}^q(Y)_\Lambda
 \simeq
 \mathrm{CH}^q(Y)_\Lambda
 \times H^q_{\mathrm{sing}}(Y(\R),\Lambda).
\end{equation}
These isomorphisms are compatible with the graded products. For a
nontrivial line bundle $\mathcal L$, the singular factor has coefficients
in the sign local system determined by $\mathcal L(\R)$; see
\cite[Theorem~3.12]{RealCycleClass2021}. The Chow--Witt correspondence
group of $X$ uses the twist $p_2^*\omega_X$. A nonzero invariant top
differential trivializes $\omega_X$, and hence this twist, so there is an
isomorphism of $\Lambda$-modules
\[
 \widetilde{\mathrm{CH}}^g
 (X\times X,p_2^*\omega_X)_\Lambda
 \simeq
 \mathrm{CH}^g(X\times X)_\Lambda
 \times
 H^g_{\mathrm{sing}}(X(\R)\times X(\R),\Lambda).
\]
\end{remark}

Combined with the calculation in Example \ref{elliptic}, we get the splittings for elliptic curves.

\begin{corollary}
    In the case of elliptic curves $X$ over $\R$, i.e., abelian varieties over $\R$ of dimension 1, let $\text{End}(X_{\C})$ be the ring of 
    integers in the quadratic imaginary extension $L=\Q(\sqrt{d})$ of $\Q$, where $d$ is square-free, and let $\text{End}(X)=\Z$. Assume $X(\R)\neq\emptyset$. We have the following splitting in $\text{SH}(\R)_{\Lambda}$ for $2\in\Lambda^{\times}$:
    \begin{equation*}
        X_+\simeq S^{0,0,+}_\Lambda\vee\mathbb{J}^+(X)\vee S^{2,1,+}_\Lambda
        \vee\bigvee^{n(X)}(S^{0,0,-}_\Lambda\vee S^{1,0,-}_\Lambda)
    \end{equation*}
    where 
    \begin{equation*}
        n(X)=
        \begin{cases}
        1, & \text{if}\: d\equiv 1 \:\text{mod}\: 4,\\
        2, & \text{if}\: d\equiv 2 \:\text{mod}\: 4,\\
        1\:\text{or}\: 2, & \text{if}\: d\equiv 3 \:\text{mod}\: 4.
        \end{cases}
    \end{equation*}
\end{corollary}
\begin{proof}
    This combines Example \ref{elliptic} with the plus--minus decomposition. It is worth mentioning that the curve splitting itself is integral: Theorem \ref{theta}, or Theorem \ref{integral}, supplies it before localization. We require $2\in\Lambda^{\times}$ here only to separate the plus and minus parts. The object $\mathbb J^+(X)$ is the plus projection of R\"ondigs' middle summand and corresponds motivically to the Jacobian part.
\end{proof}

\section{Appendix}\label{appendix}
\subsection{Appendix A: Fourier transform on abelian varieties}\label{appendixa}
Let $X$ be a smooth variety over $k$ and let $K(X)$ denote the Grothendieck group of vector bundles on $X$. There exists a natural ring homomorphism
\begin{equation*}
    \text{ch}:K(X)\to CH^*_{\mathbb{Q}}(X),
\end{equation*}
called the \emph{Chern character}. For a line bundle $L$ with associated divisor class $l=c_1(L)\in\text{CH}^1_{\mathbb{Q}}(X)$, it is given by 
\begin{equation*}
    [L]\mapsto e^{l}:=1+l+\frac{1}{2}l^2+\frac{1}{3!}l^3+\cdots.
\end{equation*}
(Note that $e^{l}$ involves only a finite sum, since $\text{CH}^i(X)=0$ for $i>\dim(X)$.)

We will require a variant of the above construction relative to a fixed base variety. To this end, let $k$ be a field, and let $S$ be a smooth quasi-projective $k$-scheme. Consider 
the category $\text{Sm}_S$ of smooth projective $S$-schemes. Let $X$ and $Y$ be two smooth projective $S$-schemes. Elements of $\text{CH}^*_{\mathbb{Q}}(X\times_S Y)$ are referred to as \emph{relative correspondences} between $X$ and $Y$.
As before we can compose correspondences. Moreover, we also consider abelian schemes over $S$ as a relative version of abelian varieties.

\begin{definition}\label{Pontryagin}
    Let $S$ be a quasi-projective smooth variety over a field $k$. Let $X$ be an abelian scheme 
    over $S$ with multiplication map $m: X\times_S X\to X$. The \emph{Pontryagin product}, or \emph{convolution product}
    $$*:\text{CH}^{*}(X)\times\text{CH}^{*}(X)\to\text{CH}^{*}(X)$$
    (relative to $S$) is the map defined by
    $$\alpha *\beta=m_{*}(p_1^{*} \alpha \cdot p_2^{*} \beta).$$
\end{definition}

Intuitively, the product $\alpha *\beta$ is obtained by adding the points on cycles representing $\alpha$
and $\beta$. Note that the Pontryagin product depends on the base variety $S$, although this dependence is 
not reflected in the notation.

We now state two lemmas concerning the Pontryagin product:

\begin{lemma}\cite{abelian}
	Let $g=\text{dim} (X/S)$. The Pontryagin product makes $\text{CH}^*(X)=\oplus_i \text{CH}^i(X)$ into a commutative
	ring for which the cycle $[e(S)]\in \text{CH}^g(X)$ given by the identity section $e(S)\subset X$ is the identity
	element.
\end{lemma}

\begin{lemma}\cite{abelian}\label{push}
    Let $f:X\to Y$ be a homomorphism of abelian schemes over $S$. Then we have $f_*(\alpha*\beta)=f_*(\alpha)*f_*(\beta)$ 
    for all $\alpha$, $\beta\in\text{CH}^*(X)$.
\end{lemma}

Following Beauville \cite{beauville2006quelques}, we are now in a position to construct the Fourier transform; see also \cite[Chapter 13]{abelian}.

\begin{definition}
	Situation as in \ref{Pontryagin}. Let $l=c_1(\mathcal{P}_X)\in\text{CH}^1(X\times_S X^t)$ be the class of the Poincar\'e bundle of
	$X$, where $X^t$ is the dual abelian scheme. We define the \emph{Fourier transform} $T$ of $X$ as the correspondence from $X$ to $X^t$ given by
	$$T=\text{ch} (\mathcal{P})=\exp (l)=1+l+\frac{1}{2!} l^2+\cdots \in \text{CH}^{*}_{\mathbb{Q}}(X\times_S X^t).$$
	We write
	$$\tau=\tau_{\text{CH}}:\text{CH}^{*}_{\mathbb{Q}}(X)\to \text{CH}^{*}_{\mathbb{Q}}(X^t)$$ 
	for the homomorphism associated to the element $[\mathcal{P}]\in K(X\times_S X^t)$. Concretely,
	$$\tau_{\text{CH}} (x)=p_{X^t, *} (e^l \cdot p^{*}_X x)=p_{X^t, *} (T \cdot p^{*}_X x) \text{ for } x\in \text{CH}^{*}_{\mathbb{Q}}(X).$$
\end{definition}

If $\tau^t$ is the Fourier transform on $X^t$, then

\begin{theorem}\cite{abelian}
    Situation as in (\ref{Pontryagin}). Let $g=\dim (X/S)$.\\
    We have $\tau^t_{\text{CH}}\circ\tau_{\text{CH}}=(-1)^g(-\text{id}_X)^*$. For all $x,y\in\text{CH}^*_{\mathbb{Q}}(X)$ we have the relations $\tau_{\text{CH}}(x*y)=\tau_{\text{CH}}(x)\cdot\tau_{\text{CH}}(y)$
    and $\tau_{\text{CH}}(x\cdot y)=(-1)^g\tau_{\text{CH}}(x)*\tau_{\text{CH}}(y)$.
\end{theorem}

As a further corollary, we obtain the following elegant result.

\begin{theorem}\cite{abelian}
	The Fourier transform of $X$ induces an isomorphism of rings
	$$\tau=\tau_{\text{CH}}:(\text{CH}^*_{\mathbb{Q}}(X), *)\xrightarrow{\thicksim} (\text{CH}^*_{\mathbb{Q}}(X^t), \cdot),$$
	where $\cdot$ and $*$ denote the intersection product and the convolution product, respectively.
\end{theorem}

This theorem provides justification for the terminology “Fourier transform.” Just as the classical Fourier transform for functions on 
the real line transforms the convolution product into the pointwise product, our Fourier transform interchanges the Pontryagin product—which may be viewed as a kind of convolution product—with the usual intersection product.

\begin{proposition}\cite{abelian}
    Let $X/S$ be an abelian scheme of relative dimension $g$. Let $\xi^t:X^t\to S$ be the structure morphism, with zero section $e^t:S\to X^t$. Then 
    we have 
    \begin{equation*}
        \tau_{\text{CH}}(1_X)=(-1)^g\cdot e_*^t(1_S)
    \end{equation*}
    in $\text{CH}^*_{\mathbb{Q}}(X^t)$.
\end{proposition}
    
Now, let $S$ be a smooth connected quasi-projective scheme of dimension $d$ over a field $k$. We consider an abelian scheme $f:X\to S$ of relative dimension $g$. 
   
If $x\in X(S)$ is a section of $f$, we define the graph class $[\Gamma_x]$ of $x$ by
\begin{equation*}
    [\Gamma_x]:=x_*[S]=[x(S)]\in\text{CH}^g_{\mathbb{Q}}(X).
\end{equation*}
In particular, $[\Gamma_e]$ is the identity element of $\text{CH}^*_{\mathbb{Q}}(X)$ for the Pontryagin product.

Next, let $i_x:=x\times 1_{X^t}:S\times_S X^t\to X\times_S X^t$, and consider the pull-back $i^*_x(l)\in \text{CH}^1_{\mathbb{Q}}(X^t)$ of the Poincar\'{e} bundle. The following
two formulas, due to Beauville \cite{beauville2006quelques}, express relations between $i^*_x(l)$ and the graph classes $[\Gamma_x]$.

\begin{lemma}\cite{abelian}\label{Beauville}
    For all $x\in X(S)$ we have 
    \begin{equation*}
         \tau([\Gamma_x])=\exp(i^*_x l) \text{ and } \tau^t(i^*_x l)=(-1)^{g+1}\sum^{g+d}_{j=1} \frac{(-1)^{j+1}}{j}\cdot ([\Gamma_x]-[\Gamma_e])^{*j}.
    \end{equation*}
\end{lemma}

We also have the multiplicative rule for graph classes under the Pontryagin product:

\begin{lemma}\cite{abelian}\label{multiplication}
    For $x,y\in X(S)$ we have $[\Gamma_x]*[\Gamma_y]=[\Gamma_{x+y}]$.
\end{lemma}

In view of Lemma (\ref{Beauville}) we now put 

\begin{equation*}
    \log([\Gamma_x]):=(-1)^{g+1}\cdot\tau^t(i^*_x l).
\end{equation*}

And we have 

\begin{corollary}\cite{abelian}\label{log}
    The map $x\mapsto\log([\Gamma_x])$ is a group homomorphism.
\end{corollary}
\begin{proof}
    This follows from the identity of formal power series $\log((1+x)(1+y))=\log(1+x)+\log(1+y)$.
\end{proof}

\subsection{Appendix B: Topology of real points}\label{topologyofrealpoints}
In this section, we study the real points of a real abelian variety $X$. In particular, we determine the number of connected components of $X(\R)$ in order to describe the splitting of $X(\R)_+$ in $\text{SH}$. If we know more information about the endomorphism ring of $X$, we obtain an explicit formula for $n(X)$ and can compute it in examples. We follow \cite{ASENS,Huisman1994}.

Let $X$ be an abelian variety of dimension $g$ over $\R$ with a rational point. Let $X(\R)^0$ denote the connected component
of the identity in the group $X(\R)$ of real points.

\begin{proposition}\cite{ASENS}\label{realabelian1}
    (i)$X(\R)^0$ is a real torus of dimension $g$.\\
    (ii)$X(\R)/X(\R)^0$ is an elementary abelian 2-group.\\
    (iii)$X(\R)\simeq (\R/\Z)^g\times(\Z/2)^d$ with $0\leq d\leq g$.
\end{proposition}
\begin{proof}
    (i)\: Since $X(\R)^0$ is a connected, compact, abelian real Lie group of dimension $g$, it must be isomorphic to the torus $(\R/\Z)^g$.\\
    (ii)\: Consider the map $\mathbb{N}:X(\C)\to X(\R)$ defined by $\mathbb{N}(P)=P+\bar{P}$, where $\bar{P}$ denotes complex conjugation. Since $\mathbb{N}$ is a continuous homomorphism and 
          $X(\C)$ is compact and connected, the image $\mathbb{N}X(\C)$ is a closed connected subgroup of $X(\R)$. Moreover, since it contains $2X(\R)$, it must have finite index and is also open. Consequently, $\mathbb{N}X(\C)=X(\R)^0$, and the quotient is annihilated by 2.\\
    (iii)\: Since $X(\R)^0$ is a divisible group, the exact sequence
          \begin{equation*}
              0\to X(\R)^0\to X(\R)\to X(\R)/X(\R)^0\to 0  
          \end{equation*}
          splits. Hence, $X(\R)\cong (\R/\Z)^g\times (\Z/2)^d$. The bound on $d$ follows from a count of the 2-torsion points: $(\Z/2)^{g+d}\cong X(\R)_2\subset X(\C)_2\cong (\Z/2)^{2g}$. 
\end{proof}

Let $n(X)=\text{Card}(X(\R)/X(\R)^0)$ be the number of connected components of $X(\R)$. Then $n(X)=2^d$, using the above notation.

We aim to study the relationship between the number of connected components of $X(\R)$ and arithmetic properties of $\text{End}(X_{\C})$. To simplify the problem, we focus on abelian varieties
$X$ over $\R$ satisfying the following three conditions:\\
(i)\:$X$ is absolutely simple, i.e., $X_\C=X\otimes\C$ contains no nontrivial complex abelian subvarieties.\\
(ii)\:$X$ admits sufficiently many complex multiplications (see \cite{OORT1973399}), i.e., the ring of endomorphisms $\text{End}(X_\C)$ of 
      $X_\C$ has rank $2\text{dim}(X)$.\\
(iii)\:$\text{End}(X_{\C})$ is a Dedekind domain.

Let $B$ be a Dedekind domain and let $A$ be a subring of $B$. Denote by $\mathcal{A}_{\R}(B/A)$ the set of isomorphism classes of absolutely
simple abelian varieties $X$ over $\R$ that admit sufficiently many complex multiplications such that
\begin{equation*}
    \text{End}(X)\cong A,\text{End}(X_\C)\cong B.
\end{equation*}

\begin{remark}\label{fact}
    The set $\mathcal{A}_{\R}(B/A)$ is nonempty if and only if the following four conditions hold:\\
    (i)\:$B$ is a commutative domain and finitely generated as a $\Z$-module.\\
    (ii)\:The field of fractions $L$ of $B$ is a totally imaginary extension of a totally real field $K$, 
          and the field extension $L/K$ has degree 2.\\
    (iii)\:$A=B\cap K$.\\
    (iv)\:$L$ admits a primitive CM type, i.e., one not induced from a proper CM subfield \cite[Proposition~3.13]{MilneCM2020}.
\end{remark}

From these conditions, it follows that $X\in\mathcal{A}_{\R}(B/A)$ if and only if $X$ is an absolutely simple abelian variety over $\R$ that 
admits sufficiently many complex multiplications and satisfies $\text{End}(X_{\C})\cong B$. Moreover, such an isomorphism automatically 
induces an isomorphism $\text{End}(X)\cong A$.

Suppose the rings $A$ and $B$ satisfy all the conditions in Remark \ref{fact}. Then $L/K$ is a Galois extension. Let $G$ denote its Galois group, and let $\sigma\in G$ denote its nontrivial element. Furthermore, let $B(G)$ be the smallest subring of $\text{End}_A(B)$,
the ring of $A$-linear endomorphisms of $B$, that contains $B$ as well as $G$.

Then, one can construct an abelian variety $X\in\mathcal{A}_{\R}(B/A)$ as follows: Choose a morphism of $\R$-algebras
\begin{equation*}
    \Phi:\C\to\R\otimes B,
\end{equation*}
whose CM type is primitive. Such a morphism $\Phi$ is called a \emph{simple complex structure}
on $\R\otimes B$. Next, choose a $B(G)$-module $M$ that is projective of rank 1 as a $B$-module. Define
\begin{equation*}
    V=\R\otimes M=(\R\otimes B)\otimes_B M,
\end{equation*}
which is a complex vector space via $\Phi$ and contains $\Lambda=1\otimes M$ as a lattice. The group $G$ acts on $V$,
where the action of $\sigma$ is anti-$\C$-linear, and $\Lambda$ is $G$-invariant. It is a standard result that there exists an absolutely simple abelian 
variety over $\R$ and a $G$-equivariant isomorphism of complex Lie groups
\begin{equation*}
    X(\C)\to V/\Lambda.
\end{equation*}
The variety $X$ admits sufficiently many complex multiplications and satisfies $\text{End}(X_\C)\cong B$, hence $X\in\mathcal{A}_{\R}(B/A)$.
Since the isomorphism class of $X$ is uniquely determined by $\Phi$ and $M$, we denote this variety $X$ by $X_{\R}(M,\Phi)$.

Conversely, if $X\in\mathcal{A}_{\R}(B/A)$, then there exists a simple complex structure $\Phi$ on $\R\otimes B$ and a $B(G)$-module $M$, projective 
of rank 1 over $B$, such that
\begin{equation*}
    X_{\R}(M,\Phi)\cong X.
\end{equation*}
More details can be found in \cite{Huisman1992RealAV}.

For technical reasons, and in accordance with Remark \ref{fact}, we will assume from now on that the rings $A$ and $B$ satisfy the conditions in Remark \ref{fact} and $B$ and $A$ are Dedekind domains.

We further introduce the following concept. Let $M$ be a $B(G)$-module. Then the \emph{ramification module} of $M$, denoted as $\mathfrak{E}(M)$, is defined as the cokernel of the canonical mapping of $B$-modules
\begin{equation*}
    B\otimes_A(M^G)\to M,
\end{equation*}
where $M^G=\{m\in M|\sigma m=m\}$. Assume $M$ to be a projective 
of rank 1 as a $B$-module. Then:
\begin{equation*}
    \mathfrak{E}(M)\cong\bigoplus^n_{i=1} B/\PP_i^{\epsilon_i},
\end{equation*}
for some $\epsilon_i\in\{0,1\}$. We may assume $M$ is a $B(G)$-submodule of $B$. Then $M=\prod\PP^{e_{\PP}}$, where the product is taken over all nonzero prime ideals of $B$ and where all but
finitely many of the integers $e_{\PP}$ are zero. Since $\sigma M=M$, there exists an ideal $\mathfrak{a}$ of $A$ such that $M=\mathfrak{a}\cdot\prod_{i=1}^{n}\PP_i^{e_i}$, where $e_i=e_{\PP_i}$.
Thus, $M^G=\mathfrak{a}\cdot\prod_{i=1}^n\p^{d_i}_i$, where $d_i=\lfloor\frac{e_i+1}{2}\rfloor$. Hence, $B\otimes_A(M^G)=\mathfrak{a}\cdot\prod^n_{i=1}B\p^{d_i}_i=\mathfrak{a}\cdot\prod^n_{i=1}\PP^{2d_i}_i$.
We may then set $\epsilon_i=2d_i-e_i$, which gives the ramification module of $M$.

Now, suppose $M$ is a $B(G)$-module that is projective of rank 1 as a $B$-module. By construction, $\mathfrak{E}(M)=0$ if and only if there exists an $A$-module $N$ such that
$B\otimes_A N\cong M$ as $B(G)$-modules.

For any $B$-module of finite length $M$, $\chi_B(M)$ is the ideal of $B$ determined by the following properties \cite{Serre1979}:\\
(i)\:$\chi_B$ is multiplicative with respect to short exact sequences of finite-length $B$-modules;\\
(ii)\:$\chi_B(B/\mathfrak{b})=\mathfrak{b}$ for any nonzero ideal $\mathfrak{b}\subset B$.\\

Let $\D$ denote the discriminant of $B$ over $A$. A nonzero prime ideal $\p$ of $A$ ramifies in $B$ if and only if $\p|\D$.

It will be convenient to enumerate the set $S$ of nonzero prime ideals $\{\p_1,...,\p_n\}$ of $A$ dividing the discriminant
$\D$ of $B$ over $A$ in such a way that
\begin{equation*}
    \{\p_1,...,\p_m\}=\{\p\in S|\:\p\:\text{divides}\:(2)\}
\end{equation*}
and
\begin{equation*}
    \{\p_1,...,\p_l\}=\{\p\in S|\:\p\:\text{divides}\:(2)\:\text{and}\:\text{ord}_{\p}(\D)\:\text{is}\:\text{even}\},
\end{equation*}
where $0\leq l\leq m$ and $\text{ord}_{\p}(\D)$ is the greatest integer $i$ such that $\p^i|\D$. Let $\PP_i$ be the unique prime 
ideal of $B$ lying over $\p_i$. Then $\PP_i^2=B\p_i$, for $i=1,...,n$.

\begin{theorem}\cite{Huisman1994}\label{components1}
    Let $X\in\mathcal{A}_{\R}(B/A)$, i.e., $X$ is an absolutely simple abelian variety over $\R$,
    admitting sufficiently many complex multiplications with $\text{End}(X)\cong A$ and $\text{End}(X_{\C})\cong B$.
    Then, the number of connected components of $X(\R)$ is equal to
    \begin{equation*}
        \prod^{m}_{i=1} 2^{a_i f_i} \bigg/ \prod^{l}_{i=1} 2^{\epsilon_i f_i},
    \end{equation*}
    where $a_i=\lfloor \frac{ord_{\p_i}(\D)}{2}\rfloor, f_i=[k(\p_i):\F_2]$ for $i=1,...,m$, and $\epsilon_i=ord_{{\PP}_i}(\chi_{B}(\mathfrak{E}(M)))$,
    for $i=1,...,l$, where $k(\p_i)$ is the residue field $A/\p_i$, and $M$ is a $B(G)$-module, projective of rank 1 as a $B$-module, such 
    that
    \begin{equation*}
        X\cong X_{\R}(M,\Phi),
    \end{equation*}
    for some simple complex structure $\Phi$ on $\R\otimes B$.
\end{theorem}
\begin{proof}
    Let $X=X_{\R}(M,\Phi)$. By definition of $X$, there is a short exact sequence of $G$-modules
    \begin{equation*}
        0\to M\to \R\otimes M\to X(\C)\to 0.
    \end{equation*}
    Taking group cohomology, and using $H^1(G,\R\otimes M)=0$, we obtain
    \begin{equation*}
        0\to M^G\to \R\otimes M^G\to X(\R)\to H^1(G, M)\to 0.
    \end{equation*}
    Thus, the number of connected components of $X(\R)$ is equal to the cardinality of $H^1(G,M)=X(\R)/X(\R)^0$.
    For any $B(G)$-module $M$, let
    \begin{equation*}
        \mathfrak{F}(M)=\{ m\in M|m+\sigma m=0\}.
    \end{equation*}
    Then $H^1(G,M)\cong\mathfrak{F}(M)/(1-\sigma)(M)$.
    Define the $B(G)$-submodule
    \begin{equation*}
        N=\prod^{n}_{i=1}\PP_i^{\epsilon_i},
    \end{equation*}
    of $B$, where $\epsilon_i=\text{ord}_{\PP_i}(\chi_{B}(\mathfrak{E}(M)))$. This implies 
    \begin{equation*}
        \mathfrak{F}(N)/(1-\sigma)(N)\cong\mathfrak{F}(M)/(1-\sigma)(M),
    \end{equation*}
    as an $A$-module. Therefore the number of connected components of $X(\R)$ is equal to $\#\mathfrak{F}(N)/(1-\sigma)(N)$.\\
    Since the canonical map $A/\p_i\to B/\PP_i$ is bijective and $\sigma$ acts as the identity on $A$, we have $(1-\sigma)(N)=(1-\sigma)(B)$,
    yielding the exact sequence
    \begin{equation*}
        0\to\mathfrak{F}(N)/(1-\sigma)(N)\to\mathfrak{F}(B)/(1-\sigma)(B)\to\mathfrak{F}(B)/\mathfrak{F}(N)\to 0.
    \end{equation*}
    The cohomology vanishes away from 2 and at unramified primes, the latter by the integral normal basis theorem. At a ramified $\p_i$ above 2, complete and write $A_0\subset B_0$. Choose a uniformizer $\theta$ of $B_0$ and write $\mathfrak{F}(B_0)=A_0\eta$. Saturation gives $\text{ord}_{\PP_i}(\eta)\in\{0,1\}$. Write $\theta-\sigma\theta=c\eta$. This generates the different, while $(1-\sigma)B_0=A_0c\eta$; hence
    \begin{equation*}
        \text{ord}_{\p_i}(\D)=2\text{ord}_{\p_i}(c)+\text{ord}_{\PP_i}(\eta),
        \qquad \chi_{A_0}(H^1(G,B_0))=\p_i^{a_i}.
    \end{equation*}
    Moreover, $\mathfrak{F}(\PP_i)=\p_i\mathfrak{F}(B_0)$ when $\text{ord}_{\p_i}(\D)$ is even, and $\mathfrak{F}(\PP_i)=\mathfrak{F}(B_0)$ when it is odd.
    Because 2 annihilates the $A$-module $\mathfrak{F}(B)/(1-\sigma)(B)$ and $B\mathfrak{F}(B)$ is square-free, we obtain:
    \begin{equation*}
        \chi_A(\mathfrak{F}(N)/(1-\sigma)(N))=\chi_A(\mathfrak{F}(B)/(1-\sigma)(B))\cdot\chi_A(\mathfrak{F}(B)/\mathfrak{F}(N))^{-1}=\prod_{i=1}^m\p_i^{a_i}\cdot\prod^{l}_{i=1}\p_i^{-\epsilon_i}.
    \end{equation*}
    Hence, the number of connected components of $X(\R)$ is 
    \begin{equation*}
        \prod^m_{i=1}2^{a_i f_i}/\prod^l_{i=1}2^{\epsilon_i f_i},
    \end{equation*}
    as claimed.
\end{proof}

\begin{remark}\cite{Huisman1994}\label{2.5}
    It is immediate that the number of connected components of $X_\R(M,\Phi)(\R)$ is independent of $\Phi$. Given the values of $a_i$ and $f_i$ as in Theorem \ref{components1}, then for any choices of $\epsilon_i\in\{0,1\}$, $i=1,...,l$, there exists an $X\in\mathcal{A}_{\R}(B/A)$ such that the number of connected 
    components of $X(\R)$ equals the number given in Theorem \ref{components1}. Take $M=\prod^l_{i=1}{\PP}_i^{\epsilon_i}$, which is a $B(G)$ module and projective of rank 1 over $B$.
    Choose any simple complex structure $\Phi$ on $\R\otimes B$. Then $X=X_{\R}(M,\Phi)$ is an element of $\mathcal{A}_{\R}(B/A)$, and Theorem \ref{components1} ensures $X(\R)$ has the desired number of connected components.
\end{remark}

\begin{remark}\cite{Huisman1994}\label{2.6}
   It is well known that if $\p$ is a nonzero prime ideal of $A$ dividing the discriminant $\D$ of $B$ over $A$, then
   \begin{equation*}
    \begin{cases}
        2\leq\text{ord}_{\p}(\D)\leq 2\text{ord}_{\p}(2)+1, & \text{if}\:\p\:|\:(2),\\
        \text{ord}_{\p}{\D}=1, & \text{if}\:\p \not|\:(2). 
    \end{cases}
   \end{equation*} 
   In particular, for any $X\in\mathcal{A}_{\R}(B/A)$, and using the notation from Theorem \ref{components1},
   \begin{equation*}
    0\leq\sum_{i=1}^m a_i f_i-\sum_{i=1}^l\epsilon_i f_i\leq\sum_{i=1}^m\text{ord}_{{\p}_i}(2)\cdot f_i\leq[K:\Q]=\text{dim}(X),
   \end{equation*}
   since $\epsilon_i=0$ or 1. Hence, by Theorem \ref{components1}, the number of connected components of $X(\R)$ is at least 1, and at most $2^{\text{dim}(X)}$. This is in accordance with the discussion at the beginning of this section.
\end{remark}

Let $\gamma$ be the integral-valued function on the set $\mathcal{A}_{\R}(B/A)$, defined by letting $\gamma(X)$ be the number of connected components of the set of 
real points $X(\R)$. Then it follows that 
\begin{equation*}
    \gamma(\mathcal{A}_{\R}(B/A))\subset\{ 1,2,...,2^g \},
\end{equation*}
where $2g=\text{rank}(B)$.

\begin{corollary}\cite{Huisman1994}\label{components2}
    Let $2g=\text{rank}(B)$. Suppose that the principal ideal $\p=(2)$ of $A$ is a prime ideal. Let $X\in\mathcal{A}_{\R}(B/A)$. Then, either,\\
    (i)\:$ord_{\p}(\D)=0$, in which case $\gamma(\mathcal{A}_{\R}(B/A))=\{ 1\}$, in particular $X(\R)$ is connected, or,\\
    (ii)\:$ord_{\p}(\D)=2$, in which case $\gamma(\mathcal{A}_{\R}(B/A))=\{ 1,2^g\}$, in particular $X(\R)$ is connected or has $2^g$ connected components, or,\\
    (iii)\:$ord_{\p}(\D)=3$, in which case $\gamma(\mathcal{A}_{\R}(B/A))=\{ 2^g\}$, in particular $X(\R)$ has $2^g$ connected components.
\end{corollary}
\begin{proof}
    These follow from Remark \ref{2.6}, Theorem \ref{components1} and Remark \ref{2.5}.
\end{proof}

\begin{example}\cite{Huisman1994}\label{elliptic}
    We now apply Corollary \ref{components2} to the case of elliptic curves over $\R$, i.e., abelian varieties over $\R$ of dimension 1. Let $B$ be the ring of 
    integers in the quadratic imaginary extension $L=\Q(\sqrt{d})$ of $\Q$, where $d$ is square-free, and let $A=\Z$. Then,
    \begin{equation*}
        \D=
        \begin{cases}
        (d), & \text{if}\: d\equiv 1 \:\text{mod}\: 4,\\
        (4d), & \text{otherwise}.
        \end{cases}
    \end{equation*}\\
    If $E$ is an elliptic curve over $\R$ with $\text{End}(E_{\C})\cong B$ then\\
    (i)\:$E(\R)$ is connected if $d\equiv 1\:\text{mod}\:4$,\\
    (ii)\:$E(\R)$ has 2 connected components if $d\equiv 2\:\text{mod}\:4$, and\\
    (iii)\:$E(\R)$ is connected or has 2 connected components otherwise.\\
    To illustrate these cases, we give an example for each.\\
    In case (i), let $d=-3$. There are exactly 2 non-isomorphic elliptic curves over $\R$ admitting
    complex multiplication by $\frac{1}{2}+\frac{1}{2}\sqrt{-3}$. Namely, the curve given by the equation
    $y^2=x^3-1$, corresponding to the lattice $\Z+\Z(\frac{1}{2}+\frac{1}{2}\sqrt{-3})$, and the curve 
    $y^2=x^3+1$, corresponding to the lattice $\Z+\Z(\frac{1}{2}+\frac{1}{6}\sqrt{-3})$. Both curves have a 
    connected set of real points, as is evident both from the equations and from the corresponding lattices.\\
    In case (ii), take $d=-2$ as an example. There are exactly 2 non-isomorphic elliptic curves over $\R$ admitting complex
    multiplication with $\sqrt{-2}$. These are given by $y^2=x(x^2-4x+2)$, corresponding to $\Z+\Z\sqrt{-2}$, and $y^2=x(x^2+4x+2)$, corresponding to $\Z+\Z\frac{1}{2}\sqrt{-2}$. In both cases, the real locus consists of two connected components.\\
    Finally, in case (iii), let $d=-1$. Again, there are also 2 non-isomorphic elliptic curves over $\R$ admitting complex multiplication
    with $\sqrt{-1}$. These are $y^2=x^3-x$, corresponding to $\Z+\Z\sqrt{-1}$, and $y^2=x^3+x$, corresponding
    to $\Z+\Z(\frac{1}{2}+\frac{1}{2}\sqrt{-1})$. The former has a real locus consisting of 2 connected components, while
    the latter has only one.
\end{example}

\begin{corollary}\cite{Huisman1994}\label{2.9}
    The following conditions are equivalent.\\
    (i)\: For any $X, Y\in\mathcal{A}_{\R}(B/A)$, the sets $X(\R)$ and $Y(\R)$ have the same number of connected components.\\
    (ii)\: $ord_{\p}(\D)$ is odd, for each prime ideal $\p$ dividing the discriminant $\D$ of $B$ over $A$.
\end{corollary}
\begin{proof}
    If $\text{ord}_{\p}(\D)$ is odd for all $\p|\D$, then $l=0$ by definition. It follows from Theorem \ref{components1} that for any $X,Y\in\mathcal{A}_{\R}(B/A)$,
    the sets $X(\R)$ and $Y(\R)$ have the same number of connected components.\\
    On the other hand, if there exists a nonzero prime ideal $\p$ dividing $\D$ with $\text{ord}_{\p}(\D)$ even, then $B$ is wildly ramified over $\p$. Hence $\p |(2)$ and $l>0$.
    By Remark \ref{2.5}, there exist $X,Y\in\mathcal{A}_{\R}(B/A)$ such that the numbers of connected components of $X(\R)$ and $Y(\R)$ are not equal.
\end{proof}

\begin{corollary}\cite{Huisman1994}\label{2.10}
    The set $X(\R)$ of real points of $X$ is connected, for any $X\in\mathcal{A}_{\R}(B/A)$, if and only if the discriminant $\D$ of
    $B$ over $A$ and the principal ideal $(2)$ are relatively prime.
\end{corollary}
\begin{proof}
    If $X(\R)$ is connected for every $X\in\mathcal{A}_{\R}(B/A)$, then $l=0$ by Corollary \ref{2.9}. Also by Remark \ref{2.6}, $a_i\geq 1$ for $i=1,...,m$, and it follows from Theorem \ref{components1}
    that $m=0$. Hence, $\D$ and $(2)$ are relatively prime.\\
    Conversely, if $\D$ and $(2)$ are relatively prime, then by definition $m=0$. Therefore, it follows from Theorem \ref{components1} that $X(\R)$ is connected for any $X\in\mathcal{A}_{\R}(B/A)$.
\end{proof}

\begin{example}\cite{Huisman1994}
    Let $k$ be an integer with $k>2$, and let $\xi$ be a primitive $k$-th root of unity. Let $B=\Z[\xi]$. Then the field of fractions $L$ of $B$ is a CM-field. Let $K$ 
    be the maximal totally real subfield of $L$, that is, $K=\Q(\xi+\xi^{-1})$, and let $A=B\cap K$. Then, the rings $A$ and $B$ satisfy conditions (i)--(iii) in Remark \ref{fact}. To study the number of connected components of $X(\R)$, 
    for any $X\in\mathcal{A}_{\R}(B/A)$, we consider two cases: whether $k$ is a power of 2 or not.\\
    First, suppose $k$ is not a power of 2. Put $r=k/\gcd(k,2)$, and let $F_r$ be the $r$-th cyclotomic polynomial. It is not difficult to see that 
    \begin{equation*}
        \prod_{i\in(\Z/r\Z)^*}(1-\xi^{2i})=F_r(1)=
        \begin{cases}
            p, & \text{if}\:r=p^a,\:p\:\text{prime},\\
            1, & \text{otherwise}.
        \end{cases}
    \end{equation*}
    Therefore, $\D=\text{N}_{L/K}((\xi-\xi^{-1}))$ and $(2)$ are relatively prime. It follows from Corollary \ref{2.10} that if $k$ is not a power of 2, the set $X(\R)$ of real points of $X$ is connected for every 
    $X\in\mathcal{A}_{\R}(B/A)$.\\
    If $k$ is a power of 2, write $k=2^a$ with $a\geq2$. There exists exactly one prime ideal $\p$ of $A$ lying over the prime ideal $(2)\subset\Z$, and $[k(\p):\mathbb{F}_2]=1$. There also exists a unique prime ideal $\PP$ of $B$ lying over $\p$, and $\PP^2=B\p$. Moreover $\text{ord}_{\p}(\D)=2$. Therefore, it follows from Theorem \ref{components1} that $X(\R)$ is either connected 
    or has 2 connected components.
\end{example}

\begin{corollary}\cite{Huisman1994}
    Let $2g=rank(B)$. Then, the set $X(\R)$ of real points of $X$ has $2^g$ connected components, for every $X\in\mathcal{A}_{\R}(B/A)$,
    if and only if $ord_{\p}(\D)=2ord_{\p}(2)+1$, for every nonzero prime $\p$ of $A$ that divides the principal ideal $(2)$.
\end{corollary}
\begin{proof}
    If $\text{ord}_{\p}(\D)=2\text{ord}_{\p}(2)+1$ for every nonzero prime ideal $\p$ of $A$ such that $\p |(2)$, then 
    \begin{equation*}
        \sum_{i=1}^m a_i f_i=\sum_{i=1}^m \text{ord}_{\p_i}(2) f_i=[K:\Q]=g.
    \end{equation*}
    According to Theorem \ref{components1}, the number of connected components of $X(\R)$, for any $X\in\mathcal{A}_{\R}(B/A)$, is equal to $2^g$.\\
    Conversely, if for every $X\in\mathcal{A}_{\R}(B/A)$, the number of connected components of $X(\R)$ is equal to $2^g$, then $l=0$ by Corollary \ref{2.9}, and 
    \begin{equation*}
        \sum^{m}_{i=1}a_i f_i=g=[K:\Q],
    \end{equation*}
    by Theorem \ref{components1}. Since we also have $\text{ord}_{\p_i}(\D)\leq2\text{ord}_{\p_i}(2)+1$ by Remark \ref{2.6}, and $a_i\leq\text{ord}_{\p_i}(2)$ by definition, it follows that
    $a_i=\text{ord}_{\p_i}(2)$ for $i=1,...,m$. Moreover, since $l=0$, we know $\text{ord}_{\p_i}(\D)$ is odd. Putting all of this together, we conclude that $\text{ord}_{\p_i}(\D)=2\text{ord}_{\p_i}(2)+1$,
    for $i=1,...,m$. Equality in the bound of Remark \ref{2.6} forces every prime above 2 to be among these.
\end{proof}

\bibliographystyle{alphaurl} 
\bibliography{Reference}

@book{MR1867354,
  address    = {Cambridge},
  author     = {Hatcher, Allen},
  isbn       = {0-521-79160-X; 0-521-79540-0},
  mrclass    = {55-01 (55-00)},
  mrnumber   = {1867354 (2002k:55001)},
  mrreviewer = {Donald W. Kahn},
  pages      = {xii+544},
  publisher  = {Cambridge University Press},
  title      = {Algebraic topology},
  username   = {mwpb479},
  year       = 2002
}

@article{10.1093/qmath/hap005,
  author   = {Röndigs, Oliver},
  title    = {{Theta characteristics and stable homotopy types of curves}},
  journal  = {The Quarterly Journal of Mathematics},
  volume   = {61},
  number   = {3},
  pages    = {351-362},
  year     = {2010},
  month    = {09},
  issn     = {0033-5606},
  doi      = {10.1093/qmath/hap005},
  url      = {https://doi.org/10.1093/qmath/hap005},
}

@article{Bachmann_20182,
  title     = {On the conservativity of the functor assigning to a motivic spectrum its motive},
  volume    = {167},
  issn      = {0012-7094},
  url       = {http://dx.doi.org/10.1215/00127094-2018-0002},
  doi       = {10.1215/00127094-2018-0002},
  number    = {8},
  journal   = {Duke Mathematical Journal},
  publisher = {Duke University Press},
  author    = {Bachmann, Tom},
  year      = {2018},
  month     = jun
}

@article{bondarko2019infinite,
  author  = {Bondarko, Mikhail V.},
  title   = {On infinite effectivity of motivic spectra and the vanishing of their motives},
  journal = {Documenta Mathematica},
  volume  = {25},
  pages   = {811--840},
  year    = {2020},
  doi     = {10.4171/DM/763},
  url     = {https://doi.org/10.4171/DM/763}
}

@book{fulton,
  title     = {Intersection Theory},
  author    = {Fulton, W.},
  isbn      = {9781461217008},
  lccn      = {97048781},
  series    = {Ergebnisse der Mathematik und ihrer Grenzgebiete},
  url       = {https://books.google.com/books?id=DiHoBwAAQBAJ},
  year      = {2012},
  publisher = {Springer New York}
}

@article{Deninger1991,
  author   = {Deninger, Ch. and Murre, Jacob},
  journal  = {Journal für die reine und angewandte Mathematik},
  pages    = {201-219},
  title    = {Motivic decomposition of abelian schemes and the {Fourier} transform.},
  url      = {http://eudml.org/doc/153379},
  volume   = {422},
  year     = {1991}
}

@book{bachmann2022milnorwitt,
  title     = {{Milnor--Witt} Motives},
  author    = {Tom Bachmann and Baptiste Calmès and Frédéric Déglise and Jean Fasel and Paul Arne Østvær},
  series    = {Memoirs of the American Mathematical Society},
  volume    = {311},
  number    = {1572},
  publisher = {American Mathematical Society},
  year      = {2025},
  doi       = {10.1090/memo/1572},
  url       = {https://doi.org/10.1090/memo/1572}
}

@article{Huisman1994,
  author  = {Huisman, J.},
  title   = {On the number of connected components of real abelian varieties that admit sufficiently many complex multiplications},
  journal = {manuscripta mathematica},
  year    = {1994},
  month   = {Dec},
  day     = {01},
  volume  = {85},
  number  = {1},
  pages   = {165-175},
  issn    = {1432-1785},
  doi     = {10.1007/BF02568191},
  url     = {https://doi.org/10.1007/BF02568191}
}

@book{Cisinski_2019,
  title     = {Triangulated Categories of Mixed Motives},
  isbn      = {9783030332426},
  issn      = {2196-9922},
  url       = {http://dx.doi.org/10.1007/978-3-030-33242-6},
  doi       = {10.1007/978-3-030-33242-6},
  journal   = {Springer Monographs in Mathematics},
  publisher = {Springer International Publishing},
  author    = {Denis-Charles Cisinski and Frédéric Déglise},
  year      = {2019}
}

@incollection{59,
  title        = {Triangulated categories of motives over a field},
  booktitle    = {Cycles, transfers, and motivic homology theories},
  series       = {Ann. of Math. Stud.},
  volume       = {143},
  year         = {2000},
  pages        = {188--238},
  publisher    = {Princeton Univ. Press, Princeton, NJ},
  organization = {Princeton Univ. Press, Princeton, NJ},
  url          = {http://www.jstor.org/stable/j.ctt7tcnh},
  author       = {Voevodsky, Vladimir}
}

@article{RONDIGS2008689,
  title    = {Modules over motivic cohomology},
  journal  = {Advances in Mathematics},
  volume   = {219},
  number   = {2},
  pages    = {689-727},
  year     = {2008},
  issn     = {0001-8708},
  doi      = {https://doi.org/10.1016/j.aim.2008.05.013},
  url      = {https://www.sciencedirect.com/science/article/pii/S0001870808001655},
  author   = {Oliver Röndigs and Paul Arne {\O}stvær}
}

@book{Mor10,
  author    = {Morel, Fabien},
  publisher = {Springer, Heidelberg},
  year      = {2012},
  month     = {07},
  pages     = {},
  title     = {$\mathbb{A}^1$-Algebraic Topology over a Field},
  volume    = {2052},
  isbn      = {978-3-642-29513-3},
  journal   = {Lecture Notes in Mathematics -Springer-verlag-},
  doi       = {10.1007/978-3-642-29514-0}
}

@article{Bachmann_20181,
  title     = {Motivic and real étale stable homotopy theory},
  volume    = {154},
  issn      = {1570-5846},
  url       = {http://dx.doi.org/10.1112/S0010437X17007710},
  doi       = {10.1112/s0010437x17007710},
  number    = {5},
  journal   = {Compositio Mathematica},
  publisher = {Wiley},
  author    = {Bachmann, Tom},
  year      = {2018},
  month     = mar,
  pages     = {883–917}
}

@book{AST_2007__314__R1_0,
  author    = {Ayoub, Joseph},
  title     = {Les six op\'erations de {Grothendieck} et le formalisme des cycles \'evanescents dans le monde motivique {(I)}},
  series    = {Ast\'erisque},
  publisher = {Soci\'et\'e math\'ematique de France},
  number    = {314},
  year      = {2007},
  mrnumber  = {2423375},
  zbl       = {1146.14001},
  language  = {fr},
  url       = {http://www.numdam.org/item/AST_2007__314__R1_0/}
}

@misc{bondarko2019torsion,
  title         = {On torsion pairs, (well generated) weight structures, adjacent $t$-structures, and related (co)homological functors},
  author        = {Mikhail V. Bondarko},
  year          = {2019},
  eprint        = {1611.00754},
  archiveprefix = {arXiv},
  primaryclass  = {math.KT},
  url           = {https://arxiv.org/abs/1611.00754}
}

@article{HU2005609,
  title    = {On the {Picard} group of the stable {$\mathbb{A}^1$}-homotopy category},
  journal  = {Topology},
  volume   = {44},
  number   = {3},
  pages    = {609-640},
  year     = {2005},
  issn     = {0040-9383},
  doi      = {https://doi.org/10.1016/j.top.2004.12.001},
  url      = {https://www.sciencedirect.com/science/article/pii/S0040938304000874},
  author   = {Po Hu}
}

@article{abeliantype,
  title     = {Remarks on motives of abelian type},
  author    = {Vial, Charles},
  journal   = {Tohoku Mathematical Journal, Second Series},
  volume    = {69},
  number    = {2},
  pages     = {195--220},
  year      = {2017},
  publisher = {Mathematical Institute, Tohoku University}
}

@phdthesis{Huisman1992RealAV,
  title   = {Real Abelian Varieties with Complex Multiplication},
  author  = {Huisman, Johannes},
  school  = {Vrije Universiteit Amsterdam},
  address = {Amsterdam},
  year    = {1992},
  url     = {https://huisman.perso.math.cnrs.fr/rech/publications/these.pdf}
}

@book{mumford1974abelian,
  title     = {Abelian Varieties},
  author    = {Mumford, David},
  series    = {Studies in mathematics},
  url       = {https://books.google.com/books?id=qcBKtgEACAAJ},
  year      = {1974},
  publisher = {Oxford University Press}
}

@article{Kunnemann1993,
  author   = {Künnemann, Klaus},
  journal  = {Inventiones mathematicae},
  number   = {1},
  pages    = {85-102},
  title    = {A {Lefschetz} decomposition for {Chow} motives of abelian schemes.},
  url      = {http://eudml.org/doc/144123},
  volume   = {113},
  year     = {1993}
}

@article{ASENS,
  author    = {Gross, Benedict H. and Harris, Joe},
  journal   = {Annales scientifiques de l'École Normale Supérieure},
  language  = {eng},
  number    = {2},
  pages     = {157-182},
  publisher = {Elsevier},
  title     = {Real algebraic curves},
  url       = {http://eudml.org/doc/82070},
  volume    = {14},
  year      = {1981}
}

@inproceedings{beauville2006quelques,
  author    = {Beauville, Arnaud},
  editor    = {Raynaud, Michel
               and Shioda, Tetsuji},
  title     = {Quelques remarques sur la transformation de {Fourier} dans l'anneau de {Chow} d'une vari{\'e}t{\'e} ab{\'e}lienne},
  booktitle = {Algebraic Geometry},
  year      = {1983},
  publisher = {Springer Berlin Heidelberg},
  address   = {Berlin, Heidelberg},
  pages     = {238--260},
  isbn      = {978-3-540-38676-6}
}

@article{OORT1973399,
  title   = {The isogeny class of a {CM}-type abelian variety is defined over a finite extension of the prime field},
  journal = {Journal of Pure and Applied Algebra},
  volume  = {3},
  number  = {4},
  pages   = {399-408},
  year    = {1973},
  issn    = {0022-4049},
  doi     = {https://doi.org/10.1016/0022-4049(73)90040-6},
  url     = {https://www.sciencedirect.com/science/article/pii/0022404973900406},
  author  = {Frans Oort}
}

@misc{abelian,
  title  = {Abelian Varieties},
  author = {Edixhoven, B. and Van der Geer, G. and Moonen, B.},
  url    = {http://van-der-geer.nl/~gerard/AV.pdf}
}

@article{R,
  issn      = {0003486X, 19398980},
  url       = {http://www.jstor.org/stable/1971109},
  author    = {A. A. Rojtman},
  journal   = {Annals of Mathematics},
  number    = {3},
  pages     = {553--569},
  publisher = {[Annals of Mathematics]},
  title     = {The Torsion of the Group of 0-Cycles Modulo Rational Equivalence},
  urldate   = {2025-02-21},
  volume    = {111},
  year      = {1980}
}

@misc{asok20110thstablea1homotopysheaf,
  title         = {The 0-th stable {$\mathbb A^1$}-homotopy sheaf and quadratic zero cycles},
  author        = {Aravind Asok and Christian Haesemeyer},
  year          = {2011},
  eprint        = {1108.3854},
  archiveprefix = {arXiv},
  primaryclass  = {math.AG},
  doi           = {10.48550/arXiv.1108.3854},
  url           = {https://arxiv.org/abs/1108.3854}
}

@article{ColliotTheleneCoray1979,
  author  = {Colliot-Th{\'e}l{\`e}ne, Jean-Louis and Coray, Daniel},
  title   = {L'{\'e}quivalence rationnelle sur les points ferm{\'e}s des surfaces rationnelles fibr{\'e}es en coniques},
  journal = {Compositio Mathematica},
  volume  = {39},
  number  = {3},
  pages   = {301--332},
  year    = {1979},
  issn    = {0010-437X},
  mrnumber = {0550646},
  url     = {https://www.numdam.org/item/CM_1979__39_3_301_0/}
}

@article{Morel04,
  author  = {Morel, Fabien},
  year    = {2004},
  month   = {12},
  pages   = {689-703},
  title   = {Sur les puissances de l'idéal fondamental de l'anneau de {Witt}},
  volume  = {79},
  journal = {Commentarii Mathematici Helvetici},
  doi     = {10.1007/s00014-004-0815-z}
}

@book{Serre1979,
  title     = {Local Fields},
  author    = {Serre, Jean-Pierre},
  series    = {Graduate Texts in Mathematics},
  volume    = {67},
  year      = {1979},
  publisher = {Springer},
  address   = {New York, NY},
  isbn      = {978-1-4757-5673-9},
  doi       = {10.1007/978-1-4757-5673-9},
  url       = {https://doi.org/10.1007/978-1-4757-5673-9}
}

@article{10.2140/akt.2017.2.357,
  author    = {Jeremy Jacobson},
  title     = {{Real cohomology and the powers of the fundamental ideal in the Witt ring}},
  volume    = {2},
  journal   = {Annals of K-Theory},
  number    = {3},
  publisher = {MSP},
  pages     = {357 -- 385},
  year      = {2017},
  doi       = {10.2140/akt.2017.2.357},
  url       = {https://doi.org/10.2140/akt.2017.2.357}
}

@article{RealCycleClass2021,
  author  = {Hornbostel, Jens and Wendt, Matthias and Xie, Heng and Zibrowius, Marcus},
  title   = {The real cycle class map},
  journal = {Annals of K-Theory},
  volume  = {6},
  number  = {2},
  pages   = {239--317},
  year    = {2021},
  doi     = {10.2140/akt.2021.6.239},
  url     = {https://doi.org/10.2140/akt.2021.6.239}
}

@book{Adams1974,
  author    = {Adams, J. Frank},
  title     = {Stable Homotopy and Generalised Homology},
  series    = {Chicago Lectures in Mathematics},
  publisher = {University of Chicago Press},
  address   = {Chicago},
  year      = {1974},
  mrnumber  = {0402720}
}

@misc{Liu24,
  author        = {Liu, Haoyang},
  title         = {On the splitting of surfaces in motivic stable homotopy category},
  year          = {2024},
  eprint        = {2406.11922},
  archiveprefix = {arXiv},
  primaryclass  = {math.AG},
  url           = {https://arxiv.org/abs/2406.11922},
  note          = {arXiv:2406.11922v2, 8 July 2025}
}

@article{Manin1968,
  author  = {Manin, Yuri I.},
  title   = {Correspondences, motifs and monoidal transformations},
  journal = {Mathematics of the USSR-Sbornik},
  volume  = {6},
  number  = {4},
  pages   = {439--470},
  year    = {1968},
  doi     = {10.1070/SM1968v006n04ABEH001070},
  url     = {https://doi.org/10.1070/SM1968v006n04ABEH001070}
}

@article{HHLM2025,
  author  = {Hasan, Junaid and Hassan, Hazem and Lin, Milton and Manivel, Marcella and McBeath, Lily and Moonen, Ben},
  title   = {Integral aspects of {Fourier} duality for abelian varieties},
  journal = {Manuscripta Mathematica},
  volume  = {176},
  eid     = {64},
  year    = {2025},
  doi     = {10.1007/s00229-025-01662-7},
  url     = {https://doi.org/10.1007/s00229-025-01662-7}
}

@article{Deglise2013,
  author  = {Déglise, Frédéric},
  title   = {Orientable homotopy modules},
  journal = {American Journal of Mathematics},
  volume  = {135},
  number  = {2},
  pages   = {519--560},
  year    = {2013},
  doi     = {10.1353/ajm.2013.0019},
  url     = {https://doi.org/10.1353/ajm.2013.0019}
}

@article{Riou2005,
  author  = {Riou, Jo{\"e}l},
  title   = {Dualit{\'e} de {Spanier--Whitehead} en g{\'e}om{\'e}trie alg{\'e}brique},
  journal = {Comptes Rendus Math{\'e}matique},
  volume  = {340},
  number  = {6},
  pages   = {431--436},
  year    = {2005},
  doi     = {10.1016/j.crma.2005.02.002},
  url     = {https://doi.org/10.1016/j.crma.2005.02.002}
}

@misc{Stacks00J9,
  author = {{Stacks Project Authors}},
  title  = {Lemma 10.32.6, Tag {00J9}},
  url    = {https://stacks.math.columbia.edu/tag/00J9},
  note   = {Accessed July 18, 2026}
}

@misc{Laumon1996,
  author        = {Laumon, G{\'e}rard},
  title         = {Transformation de {Fourier} g{\'e}n{\'e}ralis{\'e}e},
  year          = {1996},
  eprint        = {alg-geom/9603004},
  archivePrefix = {arXiv},
  url           = {https://arxiv.org/abs/alg-geom/9603004}
}

@book{Weibel2013,
  author    = {Weibel, Charles A.},
  title     = {The {K}-book: An Introduction to Algebraic {K}-theory},
  series    = {Graduate Studies in Mathematics},
  volume    = {145},
  publisher = {American Mathematical Society},
  address   = {Providence, RI},
  year      = {2013},
  url       = {https://sites.math.rutgers.edu/~weibel/Kbook.html}
}

@article{PodkopaevShinder2006,
  author  = {Podkopaev, O. B. and Shinder, E. K.},
  title   = {On the {Riemann--Roch} theorem without denominators},
  journal = {Algebra i Analiz},
  volume  = {18},
  number  = {6},
  pages   = {219--227},
  year    = {2006},
  note    = {In Russian},
  url     = {https://www.mathnet.ru/eng/aa99}
}

@article{MorelStable2005,
  author  = {Morel, Fabien},
  title   = {The stable {$\mathbb A^1$}-connectivity theorems},
  journal = {K-Theory},
  volume  = {35},
  pages   = {1--68},
  year    = {2005},
  doi     = {10.1007/s10977-005-1562-7},
  url     = {https://doi.org/10.1007/s10977-005-1562-7}
}

@misc{MilneCM2020,
  author = {Milne, J. S.},
  title  = {Complex Multiplication},
  year   = {2020},
  note   = {Version 0.10, July 14, 2020},
  url    = {https://www.jmilne.org/math/CourseNotes/CM.pdf}
}

\end{document}